\documentclass[12pt]{amsart}
\usepackage{amsmath,amsfonts,euscript,amscd,amsthm,amssymb,upref,graphics}
\usepackage[all]{xy}
\swapnumbers

\theoremstyle{plain}

\newtheorem{theorem}{Theorem}[section]
\newtheorem{proposition}[theorem]{Proposition}
\newtheorem{lemma}[theorem]{Lemma}
\newtheorem{corollary}[theorem]{Corollary}

\newtheorem{sub}{}[theorem] % This creates the counter "sub"

\newtheorem{subcorollary}	[sub]{Corollary}

\newtheorem{sublemma}		[sub]{Lemma}

\theoremstyle{definition}

\newtheorem{definition}[theorem]{Definition}

\newtheorem{parag}[theorem]{}

\newtheorem{notations}[theorem]{Notations}

\newtheorem{setup}[theorem]{Setup}
\newtheorem{Remark}[theorem]{Remark}

\newtheorem{subdefinition}[sub]{Definition}

\newtheorem{subparag}[sub]{}

\theoremstyle{remark}

{\begin{enumerate}\setlength{\itemsep}{#1}}{\end{enumerate}}
\newenvironment{enumerata}%
{\begin{enumerate}

}{\end{enumerate}}
{\begin{enumerata}\setlength{\itemsep}{#1}}{\end{enumerata}}
\newcommand{\ord}{	\operatorname{{\rm ord}}}

\newcommand{\trdeg}{	\operatorname{{\rm trdeg}}}
\newcommand{\Frac}{	\operatorname{{\rm Frac}}}
\newcommand{\Char}{	\operatorname{{\rm char}}}

\newcommand{\gr}{	\operatorname{{\rm gr}}}
\newcommand{\Gr}{	\operatorname{{\rm Gr}}}
\newcommand{\co}{	\operatorname{{\rm co}}}
\newcommand{\DEG}{	\operatorname{\mbox{\sc deg}}}
\newcommand{\Der}{	\operatorname{{\rm Der}}}

\newcommand{\setspec}[2]{\big\{\,#1\, \mid \,#2\, \big\}}

\newcommand{\powerset}{	\operatorname{\raisebox{0.7\depth}{\large $\wp$}}}
\newcommand{\pfin}{\powerset_{\text{\rm fin}}}
\newcommand{\pfino}{\powerset_{\text{\rm fin}}^*}

\newlength{\mylength}
\newcommand{\Integ}{\ensuremath{\mathbb{Z}}}
\newcommand{\Nat}{\ensuremath{\mathbb{N}}}
\newcommand{\Rat}{\ensuremath{\mathbb{Q}}}
\newcommand{\Comp}{\ensuremath{\mathbb{C}}}

\newcommand{\bk}{{\ensuremath{\rm \bf k}}}

\newcommand{\kk}[1]{\bk^{[#1]}}

\newcommand{\ggoth}{{\ensuremath{\mathfrak{g}}}}

\newcommand{\Peul}{\EuScript{P}}

\newcommand{\Ueul}{\EuScript{U}}

\newcommand{\isom}{\cong}
\renewcommand{\epsilon}{\varepsilon}
\renewcommand{\phi}{\varphi}
\renewcommand{\emptyset}{\varnothing}

\newcommand{\rien}[1]{}

\begin{document}
\renewcommand{\baselinestretch}{1.07}

%%%%%%	TOPMATTER:   %%%%%%%%%%%%%%%%%%%%%%%%%

\title{Tame and wild degree functions}

\author{Daniel Daigle}

\address{Department of Mathematics and Statistics\\
	University of Ottawa\\
	Ottawa, Canada\ \ K1N 6N5}

\email{ddaigle@uottawa.ca}

\thanks{Research supported by grant RGPIN/104976-2005 from NSERC Canada.}


{\renewcommand{\thefootnote}{}
}

\begin{abstract}
We give examples of degree functions $\deg : R \to M \cup \{-\infty\}$,
where $R$ is $\Comp[X,Y]$ or $\Comp[X,Y,Z]$ and $M$ is $\Integ$ or $\Nat$,
whose behaviour with respect to $\Comp$-derivations
$D : R \to R$ is pathological
in the sense that $\setspec{ \deg(Dx) - \deg(x) }{ x \in R\setminus \{0\} }$
is not bounded above.
We also give several general results stating that such pathologies
do not occur when the degree functions satisfy certain hypotheses.
\end{abstract}

\maketitle
  
\vfuzz=2pt

\section{Introduction}

Let $B$ be a ring and $(G,+,\le)$ a totally ordered
abelian group.  A map 
$$
\deg : B \to G \cup \{ -\infty\}
$$
is called a \textit{degree function\/} if it satisfies, for all $x,y\in B$,
\begin{enumerate}

\item $\deg(x) = -\infty$ iff $x=0$

\item $\deg(xy) = \deg x + \deg y$

\item $\deg(x+y) \le \max( \deg x, \deg y )$.

\end{enumerate}

It is easy to see that if $B$ admits a degree function
then $B$ is either the zero ring or an integral domain.
Also, if $\deg : B \to G \cup \{ -\infty\}$
is a degree function and $x,y \in B$ are such that
$\deg x \neq \deg y$, then $\deg(x+y) = \max( \deg x, \deg y )$.

\medskip
Let $B$ be an integral domain and $\deg : B \to G \cup \{ -\infty\}$
a degree function, where $G$ is a totally ordered abelian group.
Given a derivation $D : B \to B$,
$$
U = \setspec{ \deg( D x ) - \deg(x) }{ x \in B \setminus \{0\} } 
$$
is a nonempty subset of the totally ordered set $G \cup \{ -\infty \}$.
If $U$ has a greatest element, we define $\deg(D)$ to be that element;
if $U$ does not have a greatest element, we say that $\deg(D)$ is not defined.
Note that if $D$ is the zero derivation then $\deg(D)$ is defined and is equal
to $-\infty$; in fact the condition $D=0$ is equivalent to $\deg(D)=-\infty$.
Also note that, in the special case $G=\Integ$,
$\deg(D)$ is defined if and only if the set $U$ is bounded above.

\medskip
Consider the associated graded ring $\Gr(B)$, which is 
a $G$-graded integral domain determined by the pair $(B,\deg)$
(see \ref{dfjapisdjfa;kj} for details).
It is well-known that each derivation $D : B \to B$
{\it such that $\deg(D)$ is defined\/}
gives rise to a homogeneous derivation $\gr(D) : \Gr(B) \to \Gr(B)$.
The technique of replacing $D$ by $\gr(D)$,
called ``homogeneization of derivations'',
is used quite systematically in the study of $G_a$-actions on affine
algebraic varieties.
We stress that homogeneization requires prior verification that
$\deg(D)$ is defined with respect to the given degree function.
To clarify the discussion, we introduce the following notion:

\begin{definition}   \label{DefTame}
Let $A \subseteq B$ be integral domains of characteristic zero,
and let $G$ be a totally ordered abelian group.
A degree function $\deg : B \to G \cup \{-\infty\}$
is said to be {\it tame over $A$}, or {\it $A$-tame}, if it satisfies:
$$
\text{$\deg(D)$ is defined for all $A$-derivations $D : B \to B$}.
$$
If $\deg$ is not tame over $A$, we say that it is {\it wild over $A$},
or {\it $A$-wild}.
% If no confusion is likely to arise, we say that $\deg$
% is {\it tame\/} or {\it wild}, without mentioning $A$.
\end{definition}

\medskip
The present paper has two objectives:
\begin{enumerate}

\item[I.] To give examples of $\bk$-wild degree functions
$\deg : \bk[X,Y] \to \Integ \cup \{ -\infty \}$ and 
$\deg : \bk[X,Y,Z] \to \Integ \cup \{ -\infty \}$,
where $\bk$ is a field of characteristic zero;

\item[II.] to give results which state that degree functions satisfying certain
hypotheses are tame.

\end{enumerate}

\bigskip
There is a good measure of confusion in relation with degree functions.
Consider the following statement:
\begin{equation} \tag{$\times$}
\begin{minipage}{.8\textwidth}
\it 
If $B$ is an integral domain and a finitely
generated $\Comp$-algebra, 
then all degree functions on $B$ are tame over $\Comp$.
\end{minipage}
\end{equation}
Assertion~$(\times)$ is {\bf false}, as it is contradicted by either one of
\ref{dfkljasdklfja}, \ref{odfasdjflkal;ksdf} (see below).
However, $(\times)$ has been used by several authors
to justify the homogeneization of derivations.
Examples:
\cite[Pf of Lemma 1]{LML:FactsCancel},
\cite[Pf of Lemma 5]{LML:Again},
\cite[Pf of Thm 3.1]{Poloni:classifs};
% no justification :  \cite{LML:XnYZ},
in \cite{Derk:More}, a variant\footnote{Instead
of assuming that $B$ is finitely generated, the variant assumes
that $\Gr(B)$ is finitely generated. 
This variant is false: in \ref{dfkljasdklfja}, both $B$ and $\Gr(B)$
are finitely generated but $\deg$ is wild.}
of $(\times)$ is stated on page 3 and implicitly used in the proof of Prop.~2;
a (necessarily incorrect) proof of $(\times)$ is given
in \cite[6.2]{CracMaub:MLTech},
and $(\times)$ is then used to prove
the following false statement \cite[Cor.\ 6.3]{CracMaub:MLTech}:
{\it for  a $\Comp$-algebra $B$, if there exists a degree function
$\deg : B \to \Integ \cup \{-\infty\}$ such that $\Gr(B)$ is rigid,
then $B$ is rigid\/}\footnote{One says that $B$ is \textit{rigid\/} if
the only locally nilpotent derivation $D : B \to B$ is the zero derivation.}
(\ref{dfkljasdklfja} is a counterexample, as $B$ is not rigid
but $\Gr(B)=\bk[t,t^{-1}]$ is rigid).
We provide the correction:
{\it if there exists a \textit{\textbf{$\Comp$-tame}} degree function
$\deg : B \to G \cup \{-\infty\}$ such that $\Gr(B)$ is rigid,
then $B$ is rigid.}

Also, one can find many examples in the literature where
authors simply omit to raise the question whether
$\deg(D)$ is defined, as if it were a priori clear that $\deg(D)$ is
always defined.
We hope that our examples will clear-up some of this confusion.

\bigskip
Sections \ref{Sectiondfpq9wekzjkjjdf'al} and \ref{fqwekjaklsdmfklasmdlf}
prove the following facts
(the reader should compare these results
to the statement of \ref{dfj;askdjflaksj}, below).

\begin{proposition}  \label{dfkljasdklfja}
Let $\bk$ be a field of characteristic zero and $B = \bk[X,Y] = \kk2$.
Then there exists a degree function $\deg : B \to \Integ\cup\{-\infty\}$
satisfying:
\begin{enumerata}

\item $\deg(\lambda)=0$ for all $\lambda \in \bk^*$;

\item $\Gr(B) \isom \bk[t,t^{-1}]$;

\item the only $\bk$-derivation $D : B \to B$ such that $\deg(D)$ is defined
is the zero derivation.

\end{enumerata}
\end{proposition}

In the above statement and throughout this paper, we write
$A = R^{[n]}$ to indicate that $A$ is a polynomial ring in
$n$ variables over $R$.
The proof of \ref{dfkljasdklfja} is given in
Section~\ref{Sectiondfpq9wekzjkjjdf'al}.
The next fact
is the special case ``$A=\kk1$'' of \ref{dpfiouqawklej};
it shows that wild degree functions with values in $\Nat$ do exist:

\begin{proposition} \label{odfasdjflkal;ksdf}  
Let $\bk$ be an uncountable field of characteristic zero and
$B = \bk[X,Y,Z] = \kk3$.
Then there exists a degree function
$\deg: B \to \Nat \cup \{ -\infty \}$
such that $\deg(\lambda)=0$ for all $\lambda \in \bk^*$ and
with respect to which the degree of $\frac{\partial}{\partial X} : B \to B$
is not defined.
\end{proposition}

We have a similar result for $B=\kk2$, but with more restrictions on $\bk$:

\begin{proposition} \label{87837400023823897}
Let $\bk$ be a function field%
\footnote{A {\it function field\/} is a finitely generated field extension 
of transcendence degree at least~$1$.}
over an uncountable field of characteristic zero,
and let $B = \bk[X,Y] = \kk2$.
Then there exists a degree function
$\deg: B \to \Nat \cup \{ -\infty \}$
such that $\deg(\lambda)=0$ for all $\lambda \in \bk^*$ and
with respect to which the degree of $\frac{\partial}{\partial X} : B \to B$
is not defined.
\end{proposition}

Result \ref{87837400023823897} is an immediate consequence of part~(e) of
the next result,  %\ref{dfqp9023rklawej}.
which exhibits some pathologies with respect to
the process of extending degree functions:

\begin{proposition}   \label{dfqp9023rklawej}
Let $\bk_0$ be an uncountable field of characteristic zero,
$\bk_1$ a function field over $\bk_0$
and $\bk_2$ the algebraic closure of $\bk_1$.
Consider the polynomial rings $B_0 \subset B_1 \subset B_2$,
where $B_i = \bk_i[X,Y] = \bk_i^{[2]}$.
Then there exist degree functions
$$
\deg_0 : B_0 \to \Nat \cup \{ \infty \},
\quad 
\deg_1 : B_1 \to \Nat \cup \{ \infty \}
\quad \text{and} \quad 
\deg_2 : B_2 \to \Integ \cup \{ \infty \}
$$
satisfying the following conditions:

\smallskip
\begin{enumerata}
\setlength{\itemsep}{1mm}

\item if $i \le j$ then $\deg_i$ is the restriction of $\deg_j$;

\item for each $i=0,1,2$, $\deg_i( \lambda )=0$ for all $\lambda \in \bk_i^*$;

\item $\deg_0$ is determined by the grading $B_0 = \bigoplus_{i \in \Nat} R_i$
of $B_0$ defined by $X \in R_2$ and $Y \in R_3$
but, for each $i=1,2$, $\deg_i$ is not determined by a grading of $B_i$;

\item $\Gr( B_i )$ is affine over $\bk_i$ if $i \in \{0,2\}$,
but not if $i=1$;

\item $\deg_i( D_i )$ is defined if $i=0$ but not if $i \in \{1,2\}$,
where $D_i  = \frac{\partial}{\partial Y} : B_i \to B_i$.

\end{enumerata}
\end{proposition}

See \ref{dkfja;ksdjf;akj} for the proof of \ref{dfqp9023rklawej}.
The notion of a degree function {\it determined by a grading\/} is
defined in \ref{dfjapisdjfa;kj}.
It may be worthwile to state the following consequence of 
\ref{dfqp9023rklawej}:

\begin{corollary} \label{dfuqpiejrqkmkkkkd}
Let $S$ be the set of degree functions
$\deg : \Comp[X,Y] \to \Integ \cup \{-\infty\}$
satisfying 
$\deg(\lambda) = 0$ for all $\lambda \in \Comp^*$,
$\deg(X)=2$ and $\deg(Y)=3$.
Then there exist elements $d$ and $d'$ of $S$ satisfying:
$$
\textstyle
d \mid_{\Rat[X,Y]} = d' \mid_{\Rat[X,Y]},
\text{\ \ $d$ is $\Comp$-tame and $d'$ is $\Comp$-wild.}
$$
\end{corollary}

See \ref{kdfjiidfjie838} for the proof of \ref{dfuqpiejrqkmkkkkd}.

\medskip
The proof of \ref{dfkljasdklfja} is quite simple, but those
of \ref{odfasdjflkal;ksdf}--\ref{dfqp9023rklawej} are more delicate
because they involve constructing degree functions {\it with nonnegative
values\/} and which are still wild.
The crucial step is the proof, in \ref{diofua;skdmflk}, that
$\ord_t(f) \le 0$ for every nonzero element $f$ of the subring
$\bk_1[x,y]$ of $\bk_2((t))$.
The idea that this inequality could be proved by using an expansion
lemma such as \ref{dpfiuqpwejdakl} 
was inspired by past frequentations with
expansion techniques \`a la Abhyankar-Sathaye.

\bigskip
Section~\ref{Secd;lkfjpqw93ejdkla} proves an array of results which assert
that degree functions satisfying certain hypotheses are tame.
Some of those facts are summarized in the following statement,
but note that the results of Section~\ref{Secd;lkfjpqw93ejdkla} are stronger:

\begin{theorem} \label{dfj;askdjflaksj}
Suppose that $B$ is an integral domain
containing a field $\bk$ of characteristic zero.
Let $G$ be a totally ordered abelian group
and $\deg : B \to G \cup \{ -\infty\}$ a degree function.
Then, in each of the cases {\rm (a--d)} below,
% $\deg(D)$ is defined for all $\bk$-derivations $D: B \to B$ (i.e., 
$\deg$ is tame over $\bk$:
\begin{enumerata}
\setlength{\itemsep}{1.5mm}

\item  \label{diur18723hj}
$B$ is $\bk$-affine and $\deg$ is determined by some $G$-grading of $B$.

\item  \label{747438463746}
$\Gr(B)$ is $\bk$-affine and
$\setspec{ \deg(x) }{ x \in B \setminus \{0\} }$ is a well-ordered
subset of $G$.

\item  \label{sdp98123897jkd}
$\trdeg_\bk(B)<\infty$, $\Frac(B)$ is a one-dimensional function field over 
the field of fractions of the ring $\setspec{ x \in B }{ \deg(x) \le 0 }$,
and $\deg$ has values in $\Nat$.

\item  \label{p2934oq2jwdkj}
$\trdeg_\bk(B)<\infty$ and $\deg = \deg_\Delta$ for some
locally nilpotent derivation \mbox{$\Delta : B \to B$.}

\end{enumerata}
\end{theorem}

Here, $\Frac(B)$ denotes the field of fractions of $B$ and
``$\bk$-affine'' means ``finitely generated as a $\bk$-algebra''.
Assertions \eqref{diur18723hj}, \eqref{747438463746},
\eqref{sdp98123897jkd} and \eqref{p2934oq2jwdkj}
of \ref{dfj;askdjflaksj} follow from \ref{dif9w8er093409u}, 
\ref{;sdlkuf1238whd8787sd676}, \ref{d9320495823erfrj}
and \ref{doifq90348rqiojkdf}, respectively
(also note that \eqref{p2934oq2jwdkj} is a special case of
\eqref{sdp98123897jkd}).

Assertions \eqref{747438463746} and \eqref{sdp98123897jkd}
of \ref{dfj;askdjflaksj} appear to be new.
The case $G=\Integ$ of \ref{dfj;askdjflaksj}\eqref{diur18723hj}
is well known, and since the general case has the same proof we
assume that it is also known.
Assertion \ref{dfj;askdjflaksj}\eqref{p2934oq2jwdkj} appeared in
\cite[Thm 2.11, p.\ 40]{Freud:Book}, with the mention that it was unpublished
work of this author.
The material in \ref{sdp091231u23ikjd}--\ref{d9320495823erfrj}
appears to be new.
The results given in \ref{setup}--\ref{difu;awesja;lk;lj} are 
generalizations and strengthenings of known results.

\bigskip
Let us also mention that
most of the errors that we pointed out in the discussion
between \ref{DefTame} and \ref{dfkljasdklfja}
can be fixed by using the above Theorem~\ref{dfj;askdjflaksj}
in conjunction with the following observation
(\ref{dfjqio23ejqlkmskkkd} is an immediate consequence of
\ref{difuq90weukj}, below):

\begin{lemma}  \label{dfjqio23ejqlkmskkkd}
Let $B$ be an integral domain
containing a field $\bk$ of characteristic zero,
$S \subset B$ a multiplicative set,
$\DEG: S^{-1}B \to G \cup \{-\infty\}$
a degree function (where $G$ is a totally ordered abelian group)
and $\deg : B \to G \cup \{ -\infty\}$ the restriction of  $\DEG$.
If $\DEG$ is tame over $\bk$ then so is $\deg$.
\end{lemma}

\medskip
\begin{parag} \label{dfjapisdjfa;kj}
{\bf Some notations and terminologies.}
By a ``domain'', we mean an integral domain. 
If $A$ is a domain then $\Frac A$ denotes its field
of fractions.
If $A\subseteq B$ are domains then $\trdeg_A(B)$ denotes
the transcendence degree of $\Frac(B)$ over $\Frac(A)$.
The symbol $A^*$ denotes the set of units of a ring $A$.
A polynomial ring in $n$ variables over $A$ is denoted $A^{[n]}$.
A subring $A$ of a domain $B$ is said to be
{\it factorially closed in $B$} if the conditions $x,y \in B$
and $xy \in A \setminus \{0\}$ imply that $x,y \in A$.

If $A \subseteq B$ are rings then $\Der(B)$ (resp.\ $\Der_A(B)$)
is the set of derivations (resp.\  $A$-derivations) $D: B \to B$.

\medskip
Let $B$ be a domain and $G$ a totally ordered abelian group.
Then each $G$-grading $\ggoth$ of $B$ determines a degree function
$\deg_\ggoth : B \to G \cup \{ -\infty \}$ as follows. 
Let $B = \oplus_{i \in G} B_i$ be the grading $\ggoth$.
Given $x \in B$,
write $x = \sum_{i \in G} x_i$ ($x_i \in B_i$) and consider the
finite set $S_x = \setspec{ i \in G }{ x_i \neq 0 }$;
%then define $\deg_\ggoth(x) = \sup S_x$ ($=-\infty$ if $x=0$).
then define $\deg_\ggoth(x)$ to be the greatest element of 
$S_x \cup \{ -\infty \}$.
This is what we mean by a degree function ``determined by a grading''.

\medskip
Let $B$ be a domain and $\deg : B \to G \cup \{ -\infty\}$ a degree function,
where $G$ is a totally ordered abelian group.
For each $i \in G$, let
$$
B_i = \setspec{ x \in B }{ \deg(x) \le i }, \quad
B_{i^-} = \setspec{ x \in B }{ \deg(x) < i }, \quad
B_{[i]} = B_i / B_{i^-}\, .
$$
The direct sum $\Gr(B) = \bigoplus_{i \in G} B_{[i]}$
is a $G$-graded integral domain
referred to as the \textit{associated graded ring};
it is determined by $(B, \deg)$.
Note that $\Gr(B)$ comes equipped with the degree function
$\deg_\ggoth : \Gr(B) \to  G \cup \{ -\infty\}$ where $\ggoth$ denotes
the grading $\Gr(B) = \bigoplus_{i \in G} B_{[i]}$.
One also defines a set map $\gr : B \to \Gr(B)$ by $\gr(0)=0$ and,
for $x \in B \setminus \{0\}$,
$\gr(x) = x+B_{i^-} \in B_{[i]} \setminus\{0\}$, where $i = \deg(x)$.
The map $\gr$ preserves multiplication but, in general, not addition.
For all $x \in B$, $\gr(x)$ is a homogeneous element of $\Gr(B)$ and
$\deg(x) = \deg_\ggoth \big( \gr(x) \big)$.

As mentioned in the introduction,
each derivation $D : B \to B$ such that $\deg(D)$ is defined
gives rise to a homogeneous derivation $\gr(D) : \Gr(B) \to \Gr(B)$;
although this is not needed in the present paper, let us recall the definition.
Let $d=\deg(D)$.  If $d=-\infty$, set $\gr(D)=0$.
If $d \neq -\infty$ then $d \in G$ and (for each $i \in G$)
$D$ maps $B_i$ into $B_{(i+d)}$ and $B_{i^-}$ into $B_{(i+d)^-}$,
and so a map $D_{[i]} : B_{[i]} \to B_{[i+d]}$ is defined by
$x+ B_{i^-} \mapsto D(x) + B_{(i+d)^-}$;
then, given an element $y = \sum_{i \in G} y_i$ of $\Gr(B)$
(with $y_i \in B_{[i]}$),
define $\gr(D)( y ) = \sum_i D_{[i]} (y_i)$.
If $D$ is nonzero then so is $\gr(D)$, 
and if $D$ is locally nilpotent then so is $\gr(D)$.
\end{parag}

\section{Proof of \ref{dfkljasdklfja}}
\label{Sectiondfpq9wekzjkjjdf'al}

Let $\bk$ be a field of characteristic zero.

\begin{parag}  \label{dkfjapoisduf}
Consider the field $\bk((t))$ of Laurent power series over $\bk$
and the order valuation $\ord : \bk((t)) \to \Integ \cup \{ + \infty \}$.
Define 
\begin{equation}  \label{diufqiwejka3999}
\deg : \bk((t)) \to \Integ \cup \{ -\infty \}, \quad
\deg(f) = -\ord(f) \text{\ for all $f \in \bk((t))$}.
\end{equation}
Then $\deg$ is a degree function on $\bk((t))$ and it is easily verified that
the associated graded ring $\Gr \bk((t))$ is isomorphic to $\bk[ t, t^{-1}]$.
\end{parag}

\begin{parag}  \label{dfkja;sdjf;a}
Note that if $B$ is any ring such that $\bk \subseteq B \subseteq \bk((t))$
then the restriction
$$
\deg : B \to \Integ \cup \{ -\infty \}
$$
of the degree function~\eqref{diufqiwejka3999} is a degree function on $B$
satisfying $\deg(\lambda)=0$ for all $\lambda \in \bk^*$.
Also, there is an injective $\bk$-homomorphism
$\Gr B \hookrightarrow \Gr\bk((t))$.
As any ring $A$ satisfying $\bk \subseteq A \subseteq  \bk[t,t^{-1}]$
is $\bk$-affine, we see that $\Gr(B)$ is $\bk$-affine.
\end{parag}

\medskip
\begin{proof}[Proof of \ref{dfkljasdklfja}]
One can show that there exists $f(t) \in \bk((t))$ such that 
$( t, f(t), f'(t) )$ are algebraically independent over $\bk$
and $\ord f(t) \ge 0$.  Choose such an 
$f(t) = \sum_{j=0}^\infty a_j t^j$; 
let $x = t^{-1}$
and $y = f(t)$ and consider the subalgebra $B = \bk[x,y]$ of $\bk((t))$.
Note that $B = \kk2$.
Define $\deg : B \to \Integ \cup \{ -\infty \}$ as in \ref{dfkja;sdjf;a}
and note (as in \ref{dfkja;sdjf;a}) that $\Gr(B)$ is $\bk$-affine and
that $\deg(\lambda)=0$ for all $\lambda \in \bk^*$.
Note that $\deg( y - a_0 )$ is a negative integer;
as $\deg(x)=1$, it follows that
$\setspec{ \deg(h) }{ h \in B \setminus \{0\} } = \Integ$.
From this, it is easy to deduce that the natural embedding of $\Gr(B)$
into $\Gr\bk((t)) \isom \bk[t,t^{-1}]$ is actually an isomorphism:
$$
\Gr(B) \isom \bk[t,t^{-1}] .
$$
For each $n \ge 1$,
% \begin{multline*}
$$
x^ny = t^{-n} f(t) = \sum_{j=0}^\infty a_j t^{j-n}
= \sum_{j=0}^{n-1} a_j t^{j-n} + \sum_{j=n}^\infty a_j t^{j-n}
= \sum_{j=0}^{n-1} a_j x^{n-j} + \sum_{j=n}^\infty a_j t^{j-n}
$$
% \end{multline*}
so if we define $g_n \in B$ by $g_n = x^ny - \sum_{j=0}^{n-1} a_j x^{n-j}$,
% = x^ny - \sum_{k=1}^{n} a_{n-k} x^{k}$
then $g_n =   \sum_{j=n}^\infty a_j t^{j-n}$, so
$$
\text{ the set $S = \setspec{ n }{ \deg( g_n ) = 0 }$ is infinite. }
$$
We have $\frac{\partial g_n}{\partial y} = x^n$ and
\begin{multline*}
\frac{\partial g_n}{\partial x}
= nx^{n-1}y - \sum_{j=0}^{n-1} (n-j)a_j x^{n-j-1}
= nt^{1-n} \sum_{j=0}^\infty a_j t^j - \sum_{j=0}^{n-1} (n-j)a_j t^{j-n+1} \\
= \sum_{j=0}^{n-1} j a_j t^{j-n+1} + \sum_{j=n}^{\infty} n a_{j} t^{j-n+1} 
= t^{2-n} \left[
\sum_{j=0}^{n-1} j a_j t^{j-1} + \sum_{j=n}^{\infty} n a_{j} t^{j-1} \right]
= t^{2-n} \big[ \alpha_n + \epsilon_n \big]
\end{multline*}
with $\alpha_n = \sum_{j=0}^{n-1} j a_j t^{j-1}$
and $\epsilon_n =  \sum_{j=n}^{\infty} n a_{j} t^{j-1}$.
Notice that $\{  \alpha_n + \epsilon_n  \}_{ n =1 }^\infty$
is a sequence in $\bk((t))$ which converges to $f'(t)$ with respect
to the $(t)$-adic topology.

Consider the $\bk$-derivation
$D = u \frac{\partial }{\partial x} - v \frac{\partial }{\partial y} :
B \to B$, where $u,v \in B$, and assume that $\deg(D)$ is defined.
Then there exists $d \in \Integ$ satisfying
$\deg( D(g_n) ) - \deg( g_n )  \le d$ for all $n \ge 1$,
so in particular
\begin{equation}  \label{dkfjpaiosdj}
\ord \big( t^n D( g_n )  \big) \ge n -d \quad \text{for all $n \in S$}.
\end{equation}
On the other hand we have
\begin{equation}  \label{Osdpoiasjd}
t^n D( g_n )
= t^n \big( u \frac{\partial g_n }{\partial x}
- v \frac{\partial g_n}{\partial y} \big)
= t^{2} \big[ \alpha_n + \epsilon_n \big] u - v
= \big[ \alpha_n + \epsilon_n \big] x^{-2} u - v .
\end{equation}
The right hand side of \eqref{Osdpoiasjd} is a convergent sequence
in $\bk((t))$, with limit $f'(t) x^{-2} u - v$;
so the sequence $\big\{ t^n D( g_n ) \big\}_{ n=1 }^\infty$
is convergent and, by \eqref{dkfjpaiosdj}, must converge
to $0$; so 
\begin{equation}  \label{kldfj;aksdjf;a}
f'(t) x^{-2} u - v = 0.
\end{equation}
If $u\neq 0$ then \eqref{kldfj;aksdjf;a} implies
$f'(t) = x^2 v/u \in \bk(x,y) = \bk(t,f(t))$,
which contradicts our choice of $f(t)$.
So $u=0$ and, by \eqref{kldfj;aksdjf;a}, $v=0$. So $D=0$.
\end{proof}

\section{Wild degree functions with values in $\Nat$}
\label{fqwekjaklsdmfklasmdlf}

\begin{notations}
For each finite subset $S = \{ u_1, \dots, u_n \}$ of a ring $A$,
define $\mu(S) = \prod_{i=1}^n u_i \in A$
(where $\mu(\emptyset) = 1$ by convention).
If $E$ is a set, $\pfin(E)$ denotes the set of finite subsets of $E$
and $\pfino(E)$ is the set of nonempty finite subsets of $E$.
\end{notations}

\begin{lemma}  \label{dpfiuqpwejdakl}
Let $( a_i )_{i \in \Nat}$ be a sequence of elements of a ring $A$.
Define a sequence $( F_i )_{i \in \Nat}$ in $A[Y] = A^{[1]}$ by
$F_0 = Y$ and, for each $i \in \Nat$, $F_{i+1} = F_i^2 - a_i$.
Then each nonzero element of $A[Y]$ has a unique expression as a finite sum
$$
\alpha_1 \mu(S_1) + \cdots + \alpha_N \mu(S_N),
$$
where $N\ge1$, $\alpha_i \in A \setminus \{0\}$ and $S_1, \dots, S_N$
are distinct finite subsets of $\setspec{ F_i }{ i \in \Nat }$.
\end{lemma}

\begin{proof}
As $F_i$ is monic of degree $2^i$, we see that $\mu(S)$ is monic
for each finite subset $S$ of $\setspec{ F_i }{ i \in \Nat }$,
and $S \mapsto \deg( \mu(S) )$ is a bijection from the set
of finite subsets of $\setspec{ F_i }{ i \in \Nat }$ to $\Nat$.
The Lemma follows from this.
\end{proof}

\begin{lemma} \label{difuqpweijddkfkdkd}
Let $L/K$ be an extension of fields of characteristic $\neq2$
and $\Ueul$ a subset of $L$ satisfying:
\begin{enumerate}

\item[(i)] $u^2 \in K$ for all $u \in \Ueul$

\item[(ii)] $\mu(F) \notin K$ for all $F \in \pfino\Ueul$.

\end{enumerate}
Then the family $\big( \mu(F) \big)_{ F \in \pfin( \Ueul ) }$
of elements of $L$ is linearly independent over $K$.
\end{lemma}

\begin{proof}
This is certainly well-known but, in lack of a suitable reference,
we provide a proof.
We imitate the proof that if $p_1, \dots, p_n$ are distinct prime numbers
then $[\Rat(\sqrt{p_1}, \dots, \sqrt{p_n}) : \Rat] = 2^n$, see for
instance \cite{Roth}.
% The first step is to prove the implication
% \begin{equation}
% F, G \in \pfin(\Ueul),\  F \neq \emptyset,\  F \cap G = \emptyset\ \ 
% \implies\ \  \mu(F) \notin K[G] .
% \end{equation}
The first step is to prove that the set
$$
\Sigma = \setspec{ (F, G) \in \pfin(\Ueul)^2 }
{  F \neq \emptyset,\  F \cap G = \emptyset \text{ and } \mu(F) \in K[G] }
$$
is empty. 
Suppose the contrary, and choose $(F,G) \in \Sigma$ which minimizes $|G|$.
Note that $G \neq \emptyset$ by (ii);
pick $g \in G$ and let $G' = G \setminus \{g\}$.
By minimality of $|G|$,
\begin{equation}  \label{dkfjaqweijkasdkk}
(F', G') \notin \Sigma \quad \text{for all $F' \in \pfin(\Ueul)$}.
\end{equation}
Since $\mu(F) \in K[G] = K[G'][g]$ and $g^2 \in K$ by (i),
we have $\mu(F) = a+bg$ for some $a,b \in K[G']$.
Using (i) again gives $K \ni \mu(F)^2 = a^2+2abg+b^2g^2$;
since $\Char K \neq2$, $abg \in K[G']$.
If $ab \neq 0$ then  $g \in K[G']$, 
so $( \{g\} , G' ) \in \Sigma$ contradicts \eqref{dkfjaqweijkasdkk}.
If $a=0$ then $\mu( F \cup \{g\} ) = \mu(F)g = bg^2 \in K[G']$,
so $( F \cup \{g\} , G' ) \in \Sigma$
contradicts \eqref{dkfjaqweijkasdkk}.
If $b=0$ then $\mu(F)=a \in K[G']$, 
so $( F , G' ) \in \Sigma$ contradicts \eqref{dkfjaqweijkasdkk}.
These contradictions show that $\Sigma=\emptyset$.

We now prove the assertion of the Lemma, by contradiction.
Suppose that $S_1, \dots, S_n$ are  distinct elements of $\pfin(\Ueul)$
such that $\mu(S_1), \dots, \mu(S_n)$ are linearly dependent  over $K$,
and suppose that $n$ is the least natural number for which such sets exist.
Observe that $n\ge2$ and hence $\bigcup_{i=1}^n S_i \neq \bigcap_{i=1}^n S_i$.
Pick $u \in \bigcup_{i=1}^n S_i \setminus \bigcap_{i=1}^n S_i$.
Relabel the sets $S_1, \dots, S_n$ so as to have
$u \in S_1 \cap \dots \cap S_m$ and $u \notin S_{m+1} \cup \dots \cup S_n$,
and note that $1 \le m \le n-1$.
Choose $a_1, \dots a_n \in K$ not all zero such that
$\sum_{i=1}^n a_i \mu(S_i) = 0$ and note that $a_1, \dots, a_n \in K^*$
by minimality of $n$. Let $S = \bigcup_{i=1}^n S_i$.
We have
$$
u \sum_{i=1}^m a_i \mu( S_i \setminus \{u\} )
= - \sum_{i=m+1}^n a_i \mu( S_i ) ,
$$
where the two sums belong to $K[S \setminus \{u\}]$
and where $\sum_{i=1}^m a_i \mu( S_i \setminus \{u\} ) \neq0$ by minimality
of $n$.
Thus $u \in K[S \setminus \{u\}]$ and consequently
$( \{u\}, S \setminus \{u\} ) \in \Sigma$, a contradiction.
\end{proof}

\begin{parag} \label{dfuia9wefklzd} \label{Dfiuapw9e89dfajk}
Consider the following conditions on a $4$-tuple
$(\bk_0, \bk_1, \bk_2, \Ueul)$:
\begin{itemize}

\item[(i)]
$\bk_0 \subset \bk_1 \subset \bk_2$ are fields of characteristic zero
and $\bk_2$ is algebraic over $\bk_1$

\item[(ii)] $\Ueul$ is an uncountable subset of $\bk_2$

\item[(iii)] $u^2 \in \bk_1$ for all $u \in \Ueul$

\item[(iv)] $\mu(S) \notin \bk_1$ for each $S \in \pfino( \Ueul )$

\item[(v)] some element of $\Ueul$ is transcendental over $\bk_0$.

\end{itemize}
Note that, by \ref{difuqpweijddkfkdkd}, any $4$-tuple
$(\bk_0, \bk_1, \bk_2, \Ueul)$ satisfying (i--v) also satisfies:
\begin{itemize}

\item[(vi)]
the family $\big( \mu(F) \big)_{ F \in \pfin( \Ueul ) }$
of elements of $\bk_2$ is linearly independent over $\bk_1$.

\end{itemize}
\end{parag}

\begin{lemma}   \label{difuqp9023klajsd;kl}
Let $\bk_0$ be an uncountable field of characteristic zero,
$\bk_1$ a function field over $\bk_0$
and $\bk_2$ the algebraic closure of $\bk_1$.
Then there exists a subset $\Ueul$ of $\bk_2$ such that 
$(\bk_0, \bk_1, \bk_2, \Ueul)$ satisfies the conditions of
\ref{dfuia9wefklzd}.
Moreover, if $A$ is a ring such that $\bk_0 \subseteq A \subseteq \bk_1$
and $\Frac(A)=\bk_1$, then $\Ueul$ can be chosen in such a way that 
$u^2 \in A$ for all $u \in \Ueul$.
\end{lemma}

\begin{proof}
Choose a transcendence basis
$\{ t_1, \dots, t_n \}$ of $\bk_1/\bk_0$ such that
$\{ t_1, \dots, t_n \} \subset A$, let $R = \bk_0[ t_1, \dots, t_n ]$ and
$\bk = \bk_0( t_1, \dots, t_n ) = \Frac R$.
As $\bk_1/\bk_0$ is a function field, we have $n \ge1$ and it makes sense
to define $\Peul = \setspec{ t_1 - \lambda }{ \lambda \in \bk_0 }$,
which is an uncountable set of prime elements of $R$ satisfying:
$$
\text{If $p,q$ are distinct elements of $\Peul$, then $p \nmid q$ in $R$.}
$$
Choose a subset $\Ueul_1$ of $\bk_2$ such that $x \mapsto x^2$ is a
bijection from $\Ueul_1$ to $\Peul$.  Then
\begin{itemize}

\item $u^2 \in \bk$ for all $u \in \Ueul_1$

\item $\mu(S) \notin \bk$ for each $S \in \pfino( \Ueul_1 )$.

\end{itemize}
By  \ref{difuqpweijddkfkdkd},
the family $\big( \mu(F) \big)_{ F \in \pfin( \Ueul_1 ) }$
of elements of $\bk_2$ is linearly independent over $\bk$;
as $[ \bk_1 : \bk ] < \infty$, it follows that 
$E = \setspec{ F \in  \pfin( \Ueul_1 ) }{ \mu( F ) \in \bk_1 }$
is a finite set.
Thus $C = \bigcup_{F \in E} F$ is a finite subset of $\Ueul_1$
and $\Ueul = \Ueul_1 \setminus C$ is uncountable.
It is easily verified that $(\bk_0, \bk_1, \bk_2, \Ueul)$ satisfies
the conditions of \ref{dfuia9wefklzd}.
Moreover, $u^2 \in A$ for all $u \in \Ueul$.
\end{proof}

\begin{parag} \label{ThisdkjfpqwekzdkdTeiwuo}
We now fix  $(\bk_0, \bk_1, \bk_2, \Ueul)$ satisfying the requirements
of \ref{dfuia9wefklzd}.
This is in effect throughout paragraph \ref{ThisdkjfpqwekzdkdTeiwuo}.
\end{parag}

\begin{subparag}
Let $X_0, X_1, X_2, \dots$ be a countably infinite list of indeterminates
over $\bk_1$.
For each $n \in \Nat$, let
$E_n = \setspec{ f \in \bk_1[X_0, \dots, X_n ] }{ \deg_{X_n} f = 1 }$.
Note that the sets $E_n$ are pairwise disjoint;
when $f \in E_n$, we write $\co(f) \in  \bk_1[X_0, \dots, X_{n-1} ] \setminus
\{0\}$ for the coefficient of $X_n$ in $f$.

For each $p \in \Nat$, let
$\Sigma_p$ be the set of series $\xi \in \bk_1(X_0,X_1, \dots)((t))$
of the form
\begin{equation} \label{d9fq0923wedkalj}
\xi = t^{-3} ( f_p + f_{p+1} t^3 +  f_{p+2} t^6 + \cdots )
= t^{-3} \sum_{n=0}^\infty f_{p+n}t^{3n},
\end{equation}
such that $f_i \in E_i$ for all $i \ge p$.
Given $\xi \in \Sigma_p$ with notation as in \eqref{d9fq0923wedkalj}, define 
\begin{multline*}
V( \xi ) =
\big\{ (a_0, \dots, a_{p}) \in \bk_2^{p+1} \, \mid \, 
 f_p(a_0, \dots, a_p) \neq 0  \\
\mbox{\ } \qquad\qquad\qquad \qquad   \text{\ and\ }  \forall_{i > p}\,  
\deg_{X_i}( f_i(a_0, \dots, a_{p}, X_{p+1}, \dots, X_i))=1\, \big\} .
\end{multline*}
\end{subparag}

\begin{sublemma}  \label{d9fr8-q09weida}
Let $p \in \Nat$ and 
$\xi = t^{-3} \sum_{n=0}^\infty f_{p+n}t^{3n} \in \Sigma_p$ 
(where $f_i \in E_i$ for all $i \ge p$). 
Let $\xi' = \xi^2 - f_p^2t^{-6} \in \bk_1(X_0, X_1, \dots )((t))$.
Then 
\begin{enumerata}

\item $\xi' \in \Sigma_{p+1}$

\item If $(a_0, \dots, a_p) \in V( \xi )$
then there is a countable subset $C$ of $\bk_2$ such that,
for all $a_{p+1} \in \bk_2 \setminus C$, 
$(a_0, \dots, a_{p+1}) \in V( \xi' )$.

\end{enumerata}
\end{sublemma}

\begin{proof}
A straightforward calculation gives
\begin{multline*} 
\xi' = \xi^2 - f_p^2t^{-6}
= t^{-3} \big(  2f_pf_{p+1} +  (2f_pf_{p+2}+f_{p+1}^2) t^3 +  \cdots \big) \\
= t^{-3} (  g_{p+1} +  g_{p+2} t^3 +  g_{p+3} t^6 + \cdots ) 
= t^{-3} \sum_{n=0}^\infty  g_{p+1+n} t^{3n},
\end{multline*}
where 
\begin{equation} \label{dfjapsdjfakljdf}
g_{p+1+n} = \sum_{i=0}^{n+1} f_{p+i}f_{p+1+n-i}
\quad \text{for all $n \in \Nat$.}
\end{equation}
Note that 
$g_{p+1+n}$ is equal to $2f_pf_{p+1+n}$ plus a sum of terms
of the form $f_if_j$ with $i,j<p+1+n$; this shows that
\begin{equation}  \label{d9fuq9234rjak}
g_{i} \in E_{i} \text{\ \ and\ \ }   \co( g_i ) = 2f_p \co(f_i) 
\quad \text{for all $i \ge p+1$} .
\end{equation}
In particular, $\xi' \in \Sigma_{p+1}$.

Suppose that $(a_0, \dots, a_{p}) \in V( \xi )$.
Then $f_p(a_0, \dots, a_{p}) \neq 0$ and
\begin{equation}  \label{idfupqawejdfkl}
\deg_{X_i} f_i(a_0, \dots, a_{p}, X_{p+1}, \dots, X_i) = 1
\quad \text{for all $i \ge p+1$.}
\end{equation}
Let
$C = \setspec{ a_{p+1} \in \bk_2 }{ (a_0, \dots, a_{p+1}) \notin V( \xi' ) }$;
we have to show that $C$ is countable.
Note that $C$ is a countable union, $C = \bigcup_{i=p+1}^\infty C_{i}$,
where
$$
C_{p+1} = \setspec{ a_{p+1} \in \bk_2 }{ g_{p+1}(a_0, \dots, a_{p+1}) = 0 }
$$
and, for each $i \ge p+2$,
$$
C_{i} = \setspec{ a_{p+1} \in \bk_2 }
{\deg_{X_i} g_i(a_0, \dots, a_{p+1}, X_{p+2}, \dots, X_i)<1 }.
$$
Since 
$g_{p+1}(a_0, \dots, a_{p+1}) = 
2f_{p}(a_0, \dots, a_{p}) f_{p+1}(a_0, \dots, a_{p+1})$
where 
$f_{p}(a_0, \dots, a_{p}) \neq 0$ and 
$f_{p+1}(a_0, \dots, a_p, X_{p+1}) \in \bk_2[ X_{p+1} ]$ has degree $1$
by the case $i=p+1$ of \eqref{idfupqawejdfkl},
we see that $C_{p+1}$ is a finite set.
Let $i \ge p+2$.
Then $\co( g_i ) = 2f_p \co(f_i)$ by \eqref{d9fuq9234rjak}, so 
$$
\co( g_i )(a_0, \dots, a_{p}, X_{p+1}, \dots, X_{i-1})
= 2f_p(a_0, \dots, a_{p})
\co(f_i)(a_0, \dots, a_{p}, X_{p+1}, \dots, X_{i-1}).
$$
Since $f_p(a_0, \dots, a_{p}) \neq 0$ and, by \eqref{idfupqawejdfkl},
$\co(f_i)(a_0, \dots, a_{p}, X_{p+1}, \dots, X_{i-1}) \neq 0$,
we have
$$
\co( g_i )(a_0, \dots, a_{p}, X_{p+1}, \dots, X_{i-1})
\in \bk_2[ X_{p+1}, \dots, X_{i-1} ] \setminus \{0\}.
$$
Consequently, there are only finitely many $a_{p+1} \in \bk_2$
satisfying 
$$
\co( g_i )(a_0, \dots, a_{p+1}, X_{p+2}, \dots, X_{i-1})=0,
$$
or equivalently
$$
\deg_{X_i} g_i(a_0, \dots, a_{p+1}, X_{p+2}, \dots, X_i)<1.
$$
So $C_i$ is a finite set (for each $i$) and it follows that $C$ is countable.
\end{proof}

\begin{subparag}
For each $p \in \Nat$ we define a set map
(well-defined by Lemma~\ref{d9fr8-q09weida})
$$
\Sigma_p \to \Sigma_{p+1}, \quad  \xi \mapsto \xi'
$$
by setting $\xi' = \xi^2 - f_p^2t^{-6}$,
where the notation for $\xi \in \Sigma_p$ is as in \eqref{d9fq0923wedkalj}.
Define a sequence $(\xi_p)_{p \in \Nat}$ by setting
$\xi_0 = t^{-3} \sum_{n=0}^\infty X_n t^{3n} \in \Sigma_0$ and
$\xi_{p+1} = \xi_p'$ for all $p \in \Nat$.
Note that $\xi_p \in \Sigma_p$ for all $p \in \Nat$,
and let the notation be as follows:
$$
\xi_p = t^{-3} \sum_{n=0}^\infty f_{p,p+n} t^{3n}  
\quad ( f_{p,p+n} \in E_{p+n} ) .
$$
By \eqref{dfjapsdjfakljdf} we have
$f_{p+1,p+1+n} = \sum_{i=0}^{n+1} f_{p,p+i}f_{p,p+1+n-i}$
for all $p,n \in \Nat$, and in particular
\begin{equation}  \label{23894owe4uirkj}
f_{p+1,p+1} = 2 f_{p,p}f_{p,p+1} \quad \text{for all $p \in \Nat$.}
\end{equation}
\end{subparag}

\begin{sublemma}  \label{difupaw9efok}
For each $u_0 \in \Ueul$,
there exists a sequence $(a_i)_{i \in \Nat}$ of elements of $\bk_2$
satisfying the following conditions:

\begin{enumerata}

\item $a_0 = u_0$

\item $(a_0, \dots, a_p) \in V( \xi_p )$, for each $p \in \Nat$

\item $p \mapsto f_{p,p}(a_0, \dots, a_p)$
is an injective map from $\Nat$ to $\Ueul$.

\end{enumerata}
\end{sublemma}

\begin{proof}
We define  $(a_i)_{i \in \Nat}$ by induction. Define $a_0 = u_0$;
note that $(a_0) \in V( \xi_0 )$ and that $f_{0,0}(a_0)=a_0=u_0 \in \Ueul$.  

Let $p \ge 0$ and assume that  $(a_i)_{i = 0}^p$ is such that 
$a_0=u_0$, $(a_0, \dots, a_i) \in V( \xi_i )$ for all $i \in \{0,\dots, p\}$,
and $i \mapsto f_{i,i}(a_0, \dots, a_i)$ is an injective map
$\{0,\dots, p\} \to \Ueul$.

Define $e_i = f_{i,i}(a_0, \dots, a_i) \in \Ueul$, $0 \le i \le p$.
By \ref{d9fr8-q09weida}, there exists a countable set $C \subset \bk_2$
such that, for each $a_{p+1} \in \bk_2 \setminus C$,
$(a_0, \dots, a_{p+1}) \in V( \xi_{p+1} )$.
By \eqref{23894owe4uirkj} we have $f_{p+1,p+1} = 2 f_{p,p} f_{p,p+1}$, so
\begin{multline*}
f_{p+1,p+1}(a_0, \dots, a_p, X_{p+1})
= 2 f_{p,p}(a_0, \dots, a_p) f_{p,p+1}(a_0, \dots, a_p, X_{p+1})\\
= 2 e_p f_{p,p+1}(a_0, \dots, a_p, X_{p+1})  \in \bk_2[ X_{p+1} ]
\text{ is a polynomial of degree $1$,}
\end{multline*}
because $(a_0, \dots, a_p) \in V( \xi_p )$.
Consequently, $x \mapsto f_{p+1,p+1}(a_0, \dots, a_{p},x)$ is a
bijective map $\bk_2 \to \bk_2$; as $\Ueul \setminus \{e_0, \dots, e_p\}$
is uncountable, we may choose $a_{p+1} \in \bk_2 \setminus C$ such that 
$f_{p+1,p+1}(a_0, \dots, a_{p+1}) \in \Ueul \setminus \{e_0, \dots, e_p\}$.
Then $(a_0, \dots, a_{p+1}) \in V( \xi_{p+1} )$ and
% $f_{p+1,p+1}(a_0, \dots, a_{p+1}) = e_{p+1}$.
$i \mapsto f_{i,i}(a_0, \dots, a_i)$ is an injective map  $\{0,\dots, p+1 \} \to \Ueul$.
\end{proof}

\begin{subcorollary}  \label{kdfi734875oqiejk}
There exist sequences 
$(a_i)_{i \in \Nat}$ and $(e_i)_{i \in \Nat}$ of elements of $\bk_2$
satisfying:
\begin{enumerata}

\item $f_{i,i}(a_0, \dots, a_i)=e_i$ for each $i \in \Nat$;

\item $i \mapsto e_i$ is an injective map from $\Nat$ to $\Ueul$;

\item $a_0 = e_0$ is transcendental over $\bk_0$.

\end{enumerata}
\end{subcorollary}

\begin{proof}
By \ref{dfuia9wefklzd}(v), we may pick $u_0 \in \Ueul$ transcendental
over $\bk_0$; then choose $(a_i)_{i \in \Nat}$ satisfying
conditions (a--c) of \ref{difupaw9efok} and  set
$e_i = f_{i,i}(a_0, \dots, a_i)$ for each $i \in \Nat$.
\end{proof}

\begin{subdefinition}   \label{d912923prjasdkl;}
Choose sequences 
$(a_i)_{i \in \Nat}$ and $(e_i)_{i \in \Nat}$ of elements of $\bk_2$
satisfying the conditions of \ref{kdfi734875oqiejk}.
Define $x = t^{-2}$ and
$y = t^{-3} \sum_{n=0}^\infty a_n t^{3n} \in \bk_2((t))$
and, for each $i \in \{ 0, 1, 2 \}$,
consider the subring $B_i = \bk_i[x,y]$ of $\bk_2((t))$
and the degree function 
$\deg_i : B_i \to \Integ \cup \{ -\infty \}$ 
defined by $\deg_i(f) = -\ord_t(f)$, for $f \in B_i$.
Then $B_0 \subset B_1 \subset B_2$,
$\deg_i$ is the restriction of $\deg_j$ when $i \le j$, and
(for each $i=0,1,2$) $\deg_i( \lambda )=0$ for all $\lambda \in \bk_i^*$.
\end{subdefinition}

The notations of \ref{d912923prjasdkl;} are fixed until the end of
\ref{ThisdkjfpqwekzdkdTeiwuo}.
We will now show that $x,y$ are algebraically independent over $\bk_1$
and that $\deg_1$ has values in $\Nat \cup \{ -\infty \}$.
Let $\langle 2, 3 \rangle$ denote
the submonoid of $(\Integ,+)$ generated by $\{ 2, 3\}$.

\begin{sublemma}  \label{diofua;skdmflk}
$B_1 = \bk_1^{[2]}$ and $\deg_1(f) \in \langle 2,3 \rangle$
for all $f \in B_1 \setminus \{0\}$.
\end{sublemma}

\begin{proof}
Consider the subring $R$ of $\bk_1( X_0, X_1, \dots)((t))$
whose elements are the series $\sum_{i \in \Integ} f_i t^i$
satisfying $f_i \in \bk_1 [ X_0, X_1, \dots ]$ for all $i \in \Integ$
and $f_i=0$ for $i \ll 0$, and the homomorphism of $\bk_1$-algebras
$$
\textstyle
\phi : R \longrightarrow \bk_2((t)), \quad
\sum_{i \in \Integ} f_i(X_0,X_1,\dots) t^i \longmapsto
\sum_{i \in \Integ} f_i(a_0,a_1,\dots) t^i.
$$
As $\xi_p \in \Sigma_p \subset R$, we may define
$y_p = \phi( \xi_p ) \in \bk_2((t))$ for each $p \in \Nat$.
Then
$y_p = t^{-3} \sum_{n=0}^\infty f_{p,p+n}(a_0, \dots, a_{p+n}) t^{3n}$,
so in particular
\begin{equation}  \label{dofiuawiejf;kla}
y_p = e_p t^{-3} + \text{ higher powers of $t$},
\qquad \text{for all $p\in\Nat$}.
\end{equation}
Note that $y_0 = t^{-3} \sum_{n=0}^\infty a_n t^{3n}$ and 
$y_{p+1} = \phi( \xi_p^2 - f_{p,p}^2t^{-6} )
= y_p^2 - f_{p,p}(a_0, \dots, a_p)^2 t^{-6}$, so
\begin{equation}  \label{dfweuyr823389hjkd}
y_0 = y \quad \text{and} \quad
y_{p+1} = y_p^2 - e_p^2 x^{3}  \text{ for all $p \in \Nat$}.
\end{equation}
As $e_p^2 \in \bk_1$ for all $p$, this implies that 
$( y_p )_{p \in \Nat}$ is a sequence of elements of $B_1=\bk_1[x,y]$.
Consider the polynomial ring $\bk_1[X,Y] = \bk_1^{[2]}$ and
let $\pi : \bk_1[X,Y] \to B_1$ be the $\bk_1$-homomorphism sending
$X$ to $x$ and $Y$ to $y$.
Also define the sequence $( F_p )_{p \in \Nat}$ of elements of $\bk_1[X,Y]$
by $F_0 = Y$ and $F_{p+1} = F_p^2 - e_p^2 X^3$ ($p \in \Nat$).
Then \eqref{dfweuyr823389hjkd} implies that $\pi(F_p) = y_p$
for all $p \in \Nat$.

Given a finite subset $S = \{ {p_1}, \dots, {p_r} \}$ of $\Nat$
(with $p_1< \cdots < p_s$),
let $F_S = \prod_{i=1}^r F_{p_i} \in \bk_1[X,Y]$,
$y_S = \prod_{i=1}^r y_{p_i} \in \bk_1[x,y]$,
and $e_S = \prod_{i=1}^r e_{p_i} \in \bk_2$
(in particular $F_\emptyset = 1$, $y_\emptyset = 1$ and $e_\emptyset = 1$). 
Then \eqref{dofiuawiejf;kla} implies that,
given $\alpha(X) \in \bk_1[X] \setminus \{0\}$,
\begin{multline} 
\label{dpfiuqwoejkf;}
\pi( \alpha(X) F_S ) = \alpha(x) y_S  = \lambda e_S t^{m}\, +
\text{ higher powers of $t$,} \\
\text{for some $\lambda \in \bk_1^*$ and $m \in \langle -2, -3 \rangle$.}
\end{multline}

Let $G \in \bk_1[X,Y] \setminus \{0\}$.
By Lemma~\ref{dpfiuqpwejdakl},
$$
G = \alpha_1(X) F_{S_1} + \cdots + \alpha_N(X) F_{S_N},
$$
where $N \ge1$,
$\alpha_i(X) \in \bk_1[X]\setminus\{0\}$ for each $i$, and
$S_1, \dots, S_N$ are distinct finite subsets of $\Nat$.
Then \eqref{dpfiuqwoejkf;} gives
$$
\pi(G) % = \alpha_1(x) y_{S_1} + \cdots + \alpha_N(x) y_{S_N}.
= \sum_{i=1}^N  \alpha_i(x) y_{S_i}
= \sum_{i=1}^N ( \lambda_i e_{S_i} t^{m_i}\, + \text{higher powers of $t$} )
$$
for some $\lambda_1, \dots, \lambda_N \in \bk_1^*$
and $m_1, \dots, m_N \in \langle -2, -3 \rangle$.
By part (vi) of \ref{Dfiuapw9e89dfajk} together with the fact
that $p \mapsto e_p$ is injective,
the elements $e_{S_1}, \dots, e_{S_N}$ of $\bk_2$ are linearly independent
over $\bk_1$; so $\pi(G) \neq 0$ and 
$\ord_t ( \pi G ) = \min \{ m_1, \dots, m_N \} \in \langle -2, -3 \rangle$.
It follows that $\pi : \bk_1[X,Y] \to B_1$ is
bijective, so $B_1 = \bk_1^{[2]}$.
We also obtain $\deg_1(f) = -\ord_t(f) \in \langle 2, 3 \rangle$
for all $f \in B_1 \setminus\{0\}$, so the Lemma is proved.
\end{proof}

As $\bk_2/\bk_1$ is algebraic, \ref{diofua;skdmflk} implies that 
$x,y$ are algebraically independent over $\bk_2$, so:

\begin{subcorollary}
$B_i = \bk_i^{[2]}$ for $i = 0,1,2$.
\end{subcorollary}

\begin{sublemma}  \label{i23r4n2349qkdflak}
Let $A$ be a subring of $\bk_1$ satisfying $u^2 \in A$
for all $u \in \Ueul$.
Consider the subring $A[x,y]=A^{[2]}$ of $B_1 = \bk_1[x,y]$,
the degree function $\deg : A[x,y] \to \Nat \cup \{ -\infty \}$
defined by $\deg(f) = -\ord_t(f)$, and the $A$-derivation
$\frac{ \partial }{ \partial y }: A[x,y] \to A[x,y]$.
Then $\deg( \frac{ \partial }{ \partial y } )$ is not defined.
\end{sublemma}

\begin{proof}
Consider the sequence $(y_p)_{ p \in \Nat }$ of elements of $B_1$ defined
in the proof of \ref{diofua;skdmflk}.  As $y_0 = y \in A[x,y]$ and
$e_p^2 \in A$ for all $p \in \Nat$, \eqref{dfweuyr823389hjkd}
implies that $y_p \in A[x,y]$ for all $p \in \Nat$.
Also, \eqref{dofiuawiejf;kla} shows that $\deg( y_p ) = 3$ for all $p \in \Nat$.
Write $D = \frac{ \partial }{ \partial y }$, then
$D( y_{p+1} ) = D( y_p^2 - e_p^2 x^3 ) = D( y_p^2 ) = 2 y_p D( y_p )$,
so $\deg( D y_{p+1} ) = 3 + \deg( D y_p )$. Consequently,
$\deg( D y_{p} )  = 3p$ and hence
$\deg( D y_{p} ) - \deg( y_p )  = 3p-3$ for all $p \in \Nat$.
So $\deg(D)$ is not defined.
\end{proof}

For each $i = 0,1,2$, define the $\bk_i$-derivation
$D_i = \frac{\partial}{\partial y} : B_i \to B_i$.
By \ref{i23r4n2349qkdflak} we know that 
$\deg_1(D_1)$ is not defined, so in fact:

\begin{subcorollary}  \label{dfiqwekjal;ksd;}
$\deg_1(D_1)$ and $\deg_2(D_2)$ are not defined.
\end{subcorollary}

\begin{sublemma}
$\setspec{ \deg_2(f) }{ f \in B_2 \setminus \{0\} } = \Integ$
\end{sublemma}

\begin{proof}
Consider the element
$w = y^2 - a_0^2 x^3 - 2a_1 y - 2 a_0 a_2 + a_1^2$ of $\bk_2[x,y]$.
Using $y = a_0 t^{-3} + a_1 + a_2 t^3 + \cdots$ and $x = t^{-2}$, we
find
$w = 2a_0a_3 t^3 + \text{ higher powers of $t$}$,
so $\ord_t(w)>0$. 
Note that $w \neq 0$, since $x,y$ are algebraically independent over
$\bk_2$.  So $\deg(w)$ is a negative integer and consequently
$\langle 2, 3, \deg(w) \rangle = \Integ$, which proves the Lemma.
\end{proof}

\begin{sublemma}  \label{dfiouq9wemkddd}
$\Gr( B_1 )$ is not affine over $\bk_1$ and
$\Gr( B_2 )$ is affine over~$\bk_2$.
\end{sublemma}

\begin{proof}
The fact that $\Gr( B_2 )$ is affine over~$\bk_2$ follows
from $\bk_2 \subset B_2 \subset \bk_2((t))$ and $\deg_2 = -\ord_t$,
by \ref{dfkja;sdjf;a}.
Because $B_1 \nsubseteq \bk_1((t))$, we cannot apply the same argument
and show that $\Gr( B_1 )$ is affine.
In fact \ref{dfj;askdjflaksj}(b) implies that $\Gr( B_1 )$ is
not affine over $\bk_1$, because 
$\deg_1$ has values in $\Nat$ (\ref{diofua;skdmflk})
and $\deg_1( D_1 )$ is not defined (\ref{dfiqwekjal;ksd;}).
\end{proof}

The fact that $a_0$ is transcendental over $\bk_0$
(cf.\ \ref{kdfi734875oqiejk} and \ref{d912923prjasdkl;})
played no role up to this point.  It is needed for the following:

\begin{sublemma}  \label{df9q12903aokss}
Let $\ggoth$ be the $\Nat$-grading
$B_0 = \bk_0[x,y] = \bigoplus_{i \in \Nat} R_i$ of $B_0$
defined by the conditions $R_0 = \bk_0$, $x \in R_2$ and $y \in R_3$.
Then $\deg_0$ is the degree function determined by~$\ggoth$.
Consequently, $\deg_0( D_0 )$ is defined
and $\Gr( B_0 )$ is affine over~$\bk_0$.
\end{sublemma}

\begin{proof}
For each $i,j \in \Nat$, 
$$
x^iy^j = a_0^j t^{-2i-3j} + \text{ higher powers of $t$},
$$
and $a_0$ is transcendental over $\bk_0$.
It easily follows that if $S$ is a nonempty finite subset of $\Nat^2$
and $( \lambda_{ij} )_{ (i,j) \in S }$ is a family of elements
of $\bk_0 \setminus \{0\}$, then 
$$
\ord_t \bigg( \sum_{(i,j) \in S } \lambda_{ij} x^iy^j \bigg)
= \min \setspec{ -2i-3j }{ (i,j) \in S },
$$
or equivalently,
$\deg_0 \big( \sum_{(i,j) \in S } \lambda_{ij} x^iy^j \big)
= \max \setspec{ 2i+3j }{ (i,j) \in S }$.
So $\deg_0$ is the degree function determined by $\ggoth$.
A straightforward calculation shows that
$\deg_0( D_0 )$ is defined and is equal to $-3$
(alternatively, $\deg_0( D_0 )$ is defined by \ref{dfj;askdjflaksj}).
Since $\deg_0$ is determined by a grading of $B_0$,
we have $\Gr(B_0) \isom B_0$, so $\Gr(B_0)$ is affine.
\end{proof}

\begin{parag}  \label{dkfja;ksdjf;akj}
{\bf Proof of \ref{dfqp9023rklawej}.}
Let $\bk_0$ be an uncountable field of characteristic zero,
$\bk_1$ a function field over $\bk_0$
and $\bk_2$ the algebraic closure of $\bk_1$.
By \ref{difuqp9023klajsd;kl}, there exists a set $\Ueul$ such that
$(\bk_0, \bk_1, \bk_2, \Ueul)$ satisfies the requirements
of \ref{dfuia9wefklzd};
then all results of paragraph \ref{ThisdkjfpqwekzdkdTeiwuo} are valid
when applied to $(\bk_0, \bk_1, \bk_2, \Ueul)$.
Define the degree functions $\deg_i$ ($i=0,1,2$) as in \ref{d912923prjasdkl;}
and note that, by \ref{diofua;skdmflk},
$\deg_0$ and $\deg_1$ have values in $\Nat \cup \{ -\infty \}$.
Assertions (a) and (b) of \ref{dfqp9023rklawej} are clear,
and (c), (d), (e) follow from 
\ref{dfiqwekjal;ksd;}, \ref{dfiouq9wemkddd} and \ref{df9q12903aokss}
(note that, for each $i=1,2$, $\deg_i$ cannot be determined
by a grading of $B_i$ because that would imply that $\deg_i(D_i)$ is 
defined, by \ref{dfj;askdjflaksj}). \hfill \qedsymbol
\end{parag}

\begin{corollary}  \label{dpfiouqawklej}
Let $A$ be a domain which contains 
an uncountable field $\bk$ of characteristic zero,
and such that $\Frac(A)$ is a function field over $\bk$.
Consider $A[X,Y] = A^{[2]}$ and the $A$-derivation
$\frac{ \partial }{ \partial Y} : A[X,Y] \to A[X,Y]$.
Then there exists a degree function 
$\deg : A[X,Y] \to \Nat \cup \{-\infty\}$
such that $\deg(a) = 0$ for all $a \in A \setminus \{0\}$,
and such that 
$\deg( \frac{\partial}{\partial Y} )$ is not defined.
\end{corollary}

\begin{proof}
Let $\bk_0=\bk$, $\bk_1 = \Frac(A)$ and $\bk_2$ the algebraic
closure of $\bk_1$.
By \ref{difuqp9023klajsd;kl}, there exists a set $\Ueul$ such that
$(\bk_0, \bk_1, \bk_2, \Ueul)$ satisfies the requirements of
\ref{dfuia9wefklzd} and $u^2 \in A$ for all $u \in \Ueul$.
So we are done by \ref{i23r4n2349qkdflak}.
\end{proof}

\begin{parag}  \label{kdfjiidfjie838} 
{\bf Proof of \ref{dfuqpiejrqkmkkkkd}.}
There exist an uncountable field $\bk_0$ of characteristic zero
and a function field $\bk_1$ over $\bk_0$ such that the algebraic
closure of $\bk_1$ is $\Comp$.
Then the triple $(\bk_0, \bk_1, \bk_2=\Comp)$ satisfies the
hypothesis of \ref{dfqp9023rklawej}.
Let us denote by $d'$ the degree function
$\deg_2 : B_2 = \Comp[X,Y] \to \Integ\cup\{-\infty\}$
given by \ref{dfqp9023rklawej}.
Also consider the grading of $\Comp[X,Y]$ defined by stipulating that
$X,Y$ are homogeneous of degrees $2$ and $3$ respectively,
and let $d : \Comp[X,Y] \to \Integ\cup\{-\infty\}$
be the degree function determined by this grading.
Using \ref{dfqp9023rklawej}, it is easily verified that
$d \mid_{\Rat[X,Y]} = d' \mid_{\Rat[X,Y]}$ and that $d'$ is wild over $\Comp$.
By \ref{dfj;askdjflaksj}\eqref{diur18723hj}, $d$ is tame over $\Comp$.
So $d$ and $d'$ satisfy the desired conditions.
\hfill \qedsymbol
\end{parag}

\section{Some positive results}
\label{Secd;lkfjpqw93ejdkla}

We prove several results which assert
that degree functions satisfying certain hypotheses are tame.
The main results are \ref{dkfjaiosdfupaio}, \ref{NEW_dofiq-290348rakj}
and \ref{d9320495823erfrj}.

\begin{setup} \label{setup}
Throughout Section~\ref{Secd;lkfjpqw93ejdkla} we consider a triple
$(B,G,\deg)$ where
$B$ is a domain of characteristic zero,
$G$ be a totally ordered abelian group
and $\deg : B \to G \cup \{ -\infty\}$ a degree function.
\end{setup}

\begin{parag}  \label{dkfjasiodufpao}
Let $(B,G,\deg)$ be as in \ref{setup}.
If $D : B \to B$ is a derivation, one defines an auxiliary map
$\delta_D : B \to G \cup \{ -\infty \}$ by
$\delta_D(0) = -\infty$ and, given $x\in B\setminus\{0\}$,
$\delta_D(x) = \deg( Dx ) - \deg x$.  Note that
$$
\deg( Dx ) = \delta_D(x) + \deg x, \qquad\textrm{for all } x\in B.
$$
We also define $\delta_D( S ) \in G \cup \{ -\infty \}$
for certain subsets $S$ of $B$.
If $S$ is a nonempty subset of $B$ such that the subset
$U_S = \setspec{\delta_D(x)}{x\in S}$ of $G \cup \{ -\infty \}$
has a greatest element $M$, we define $\delta_D(S) = M$.
If $U_S$ does not have a greatest element, we leave $\delta_D(S)$ undefined.
We also define $\delta_D(\emptyset) = -\infty$.
Note in particular that $\delta_D(S)$ is defined for every finite
subset $S$ of $B$.
If $S_1, S_2$ are subsets of $B$ then the equality
``$\delta_D(S_1) = \delta_D(S_1)$'' is to be understood as meaning:
either both $\delta_D(S_1)$ and $\delta_D(S_1)$ are undefined, or
both are defined and are equal to the same element of $G \cup \{ -\infty \}$.
We also observe that the equality $\deg(D) = \delta_D(B)$ always holds
(i.e., either both sides are undefined, or both sides are defined
and are equal to the same element of $G \cup \{ -\infty \}$).

Define the transitive relation $\preceq_D$ on the powerset $\powerset(B)$
of $B$ by declaring that, for $S,S' \in \powerset(B)$,
$$
S \preceq_D S'\ \iff\  \forall_{ s \in S }\, \exists_{ s' \in S' }\,
\delta_D(s) \le \delta_D(s') .
$$
Then it is clear that
\begin{equation}  \label{Dkjfa;sdjf}
S \preceq_D S' \text{ and } S' \preceq_D S \ \implies\  
\delta_D(S) = \delta_D(S').
\end{equation}
Noting that $S \subseteq S'$ implies $S \preceq_D S'$,
we obtain the following useful special case of \eqref{Dkjfa;sdjf}:
\begin{equation}  \label{dlkjfasidiiii}
S \subseteq S' \text{ and } S' \preceq_D S \ \implies\  
\delta_D(S) = \delta_D(S').
\end{equation}
\end{parag}

\begin{definition}  \label{odfuipq90w389j}
Let $(B,G,\deg)$ be as in \ref{setup}.
By a \textit{$0$-subring\/} of $B$ we mean a subring $Z$
of $B$ such that $\deg(x)=0$ for all $x \in Z \setminus \{0\}$.
\end{definition}

\begin{lemma}  \label{dklfja;skldjf}
Let $(B,G,\deg)$ be as in \ref{setup}.
Let $D: B\to B$ be a derivation and $x_1,\dots,x_n\in B$.

\begin{enumerate}

\item $\delta_D(x_1x_2\cdots x_n)
\le \displaystyle \max_{1\le i\le n} \delta_D(x_i)$.

\item If $\deg( x_1+\dots+x_n ) = \displaystyle \max_{1\le i\le n} \deg(x_i)$,
then $ \delta_D( x_1+\dots+x_n )
\le \displaystyle \max_{1\le i\le n} \delta_D(x_i) $.

\item If $x_1, \dots, x_n \in Z$ for some $0$-subring $Z$ of $B$,
then $ \delta_D( x_1+\dots+x_n )
\le \displaystyle \max_{1\le i\le n} \delta_D(x_i) $.

%  \item If there exists $d\in G$ such that $\forall_i\, \deg(x_i)=d$,
%  then
%  $$
%  \sum_{1\le i\le n} \gr(x_i) = \begin{cases}
%  \gr( x_1+\dots+x_n ), & \text{if }  \deg( x_1+\dots+x_n ) =  d, \\
%  0, & \text{if }  \deg( x_1+\dots+x_n ) <  d.
%  \end{cases}
%  $$

\end{enumerate}
\end{lemma}

\begin{proof}
We write $\delta = \delta_D$.
Given $x,y\in B\setminus\{0\}$,
\begin{align*}
\delta(xy) &= \deg (D(xy))-\deg (xy) = \deg (yDx+xDy)-\deg (xy) \\
&\le  \max(\deg (yDx), \deg (xDy)) - \deg (xy) \\
&= \max(\deg (Dx)+\deg (y), \deg (Dy)+\deg (x)) - \deg (xy) =
\max(\delta(x),\delta(y));
\end{align*}
assertion~(1) follows by induction.

If $\deg( x_1+\dots+x_n ) = \max_{1\le i\le n} \deg(x_i)$ then
\begin{align*}
\deg D({\textstyle \sum_i} x_i)
= \deg ({\textstyle \sum_i} D x_i)
\le \max_i (\deg Dx_i)
= \max_i (\delta(x_i) + \deg x_i) \\
\le \max_i \delta(x_i) +\max_i \deg x_i
= \max_i \delta(x_i) +\deg({\textstyle \sum_i} x_i),
\end{align*}
so assertion~(2) holds.  Assertion~(3) immediately follows.
% Assertion~(4) is well-known, and easy.
% 
% Assume that $\deg(x_i) = d$  holds for each $i$  ($1\le i\le n$)
% and let $\pi : B_d \to B_d/B_{d^-}$ be the canonical epimorphism,
% where $B_d = \setspec{ x \in B }{ \deg(x) \le d }$
% and $B_{d^-} = \setspec{ x \in B }{ \deg(x) < d }$.
% Then
% \begin{multline*}
% \sum_{1\le i\le n} \gr(x_i)
% = \sum_{1\le i\le n} \pi(x_i)\\
% = \pi \bigg( \sum_{1\le i \le n} x_i \bigg)
% = \begin{cases}
% \gr( x_1+\dots+x_n ), & \text{if }  \deg( x_1+\dots+x_n ) =  d, \\
% 0, & \text{if }  \deg( x_1+\dots+x_n ) <  d.
% \end{cases}
% \end{multline*}
% So assertion~(4) holds.
\end{proof}

\begin{lemma}  \label{difjasdkjfa;klsd}
Let $(B,G,\deg)$ be as in \ref{setup} and
let $A \subseteq Z$ be $0$-subrings of $B$.
Suppose that $S$ is a subset of $Z$ such that $Z$ is algebraic
over $A[S]$.
Then, for all $D \in \Der_A(B)$, $\delta_D(Z) = \delta_D(S)$.
\end{lemma}

\begin{proof}
Let $D \in \Der_A(B)$ and let $\delta = \delta_D$.
Consider a product
\begin{equation}  \label{difjaposiduf;}
\mu = a x_1 \cdots x_n
\quad \text{(with $a \in A$ and $x_1, \dots, x_n \in S$)} .
\end{equation}
As $\delta(a) = -\infty$, we have
$\delta(a x_1 \cdots x_n) \le
\max \big( \delta(a), \delta(x_1), \dots, \delta(x_n) \big)
= \max_i \delta(x_i)$ by \ref{dklfja;skldjf}, so
$\delta(\mu) \le \delta(s)$ for some $s \in S$.
Now consider an element $\xi \in A[S]$.
Then $\xi$ is a finite sum, $\xi = \mu_1 + \cdots + \mu_m$,
where each $\mu_i$ is a product of the form \eqref{difjaposiduf;};
so, for each $i \in \{ 1, \dots, m \}$, there exists $s_i \in S$
such that $\delta( \mu_i ) \le \delta(s_i)$;
consequently we may choose $s \in S$ such that 
$\delta( \mu_i ) \le \delta(s)$ holds for all $i \in \{ 1, \dots, m \}$.
As $\mu_1, \dots, \mu_m \in Z$, part~(3) of \ref{dklfja;skldjf}
gives $\delta( \xi ) \le \max_i \delta( \mu_i )$,
so $\delta( \xi ) \le \delta( s )$.
This shows that $A[S] \preceq_D S$.

Let $b \in Z \setminus \{0\}$.
As $b$ is algebraic over $R = A[S]$, we may choose
a polynomial $\Phi(T) = \sum_i r_iT^i \in R[T]\setminus\{0\}$
(where $T$ is an indeterminate and $r_i \in R$)
of minimal degree such that $\Phi(b)=0$.
Then $0=D\Phi(b)=\Phi^{(D)}(b)+\Phi'(b)Db$ and (using $\Char B=0$)
$\Phi'(b)\in Z \setminus\{0\}$ imply
$\deg(\Phi^{(D)}(b)) = \deg(Db) = \delta(b)$. Now 
$\Phi^{(D)}(b)=\sum_i D(r_i)b^i$, so
$$
\delta(b) = \deg(\Phi^{(D)}(b))
= \deg \big( \sum_i D(r_i)b^i \big)
\le \max_i \deg \big( D(r_i)b^i \big)
$$
 and
$\deg(D(r_i)b^i) = \deg(Dr_i)=\delta(r_i)$ for each $i$,
so $\delta(b) \le \max_i \delta(r_i)$.
It follows that there exists $r \in A[S]$ such that $\delta(b) \le \delta(r)$,
i.e., we have shown that $Z \preceq_D A[S]$.
We get $Z \preceq_D S$ by transitivity
and $\delta(Z) = \delta(S)$ by assertion~\eqref{dlkjfasidiiii}
of \ref{dkfjasiodufpao}.
\end{proof}

%  \begin{corollary}  \label{dklfja;sdkfj}
%  Let $(B,G,\deg)$ be as in \ref{setup},
%  let $A$ be a $0$-subring of $B$ % over which $B$ has finite transcendence degree
%  and let $D \in \Der_A(B)$.
%  Then $\delta_D(Z)$ is defined for all $0$-subrings $Z$ of $B$ satisfying
%  $A \subseteq Z$ and $\trdeg_A(Z) < \infty$.
%  \end{corollary}
%  
%  \begin{proof}
%  We can choose a finite subset $S$ of $Z$ such that $Z$ is algebraic
%  over $A[S]$.  Then $\delta_D(S)$ is defined and, by \ref{difjasdkjfa;klsd},
%  $\delta_D(Z) = \delta_D(S)$; so $\delta_D(Z)$ is defined.
%  \end{proof}

\begin{proposition}  \label{dkfjaiosdfupaio}
Let $G$ be a totally ordered abelian group,
$B = \bigoplus_{i\in G} B_i$ a $G$-graded integral domain of
characteristic zero and
$\deg : B \to G \cup \{ -\infty \}$ the degree function determined
by the grading.
Assume that $B$ is finitely generated as a $B_0$-algebra and 
let $A$ be a subring of $B_0$ satisfying $\trdeg_A(B_0) < \infty$.
Then $\deg$ is tame over $A$.
% Then $\deg(D)$ is defined for all $D \in \Der_A(B)$.

More precisely,
given any choice of $z_1,\dots,z_m\in B_0$
and homogeneous $x_1,\dots,x_n\in B$
such that $B_0$ is algebraic over $A[ z_1, \dots, z_m ]$ and 
$ B = B_0 [ x_1, \dots, x_n ] $,
$$
\deg(D)= \max \{
\delta_D(z_1), \dots, \delta_D(z_m),  \delta_D(x_1), \dots, \delta_D(x_n) \}
\quad \text{for all $D \in \Der_A(B)$.}
$$
\end{proposition}

\begin{proof}
Let $z_1,\dots,z_m\in B_0$ and $x_1,\dots,x_n\in B$ be as in the statement.
Let $D \in \Der_A(B)$ and let $\delta=\delta_D$.
Define
$M = \max \{
\delta_D(z_1), \dots, \delta_D(z_m),  \delta_D(x_1), \dots, \delta_D(x_n) \}$
(so $M \in G \cup \{ -\infty \}$).
It suffices to show that $\delta(x) \le M$ for all $x \in B \setminus\{0\}$.
Indeed, if this is true then $\deg(D)=M$.

Result \ref{difjasdkjfa;klsd} (with $S = \{ z_1, \dots, z_m \}$ and $Z = B_0$)
implies that $\delta(B_0) = \max_{1 \le i \le m} \delta( z_i )$,
so $\delta(b) \le M$ certainly holds for all $b \in B_0$.

Let $x \in B \setminus \{0\}$. 
Then $x$ is a finite sum, $x = h_1 + \cdots + h_m$,
where each $h_i$ is homogeneous and $\deg(h_1)< \dots < \deg(h_m)$.
Then $\deg( h_1 + \cdots + h_m ) = \max_i \deg( h_i )$,
so part~(2) of \ref{dklfja;skldjf} implies that 
$\delta( h_1 + \cdots + h_m ) \le \max_i \delta( h_i )$.
So it's enough to show that $\delta( h_i ) \le M$ for all $i$,
i.e., we may assume that $x$ is homogeneous.

Suppose that $x \in B_d \setminus \{0\}$, for some $d \in G$.
Then $x$ is a finite sum, $x = \mu_1 + \cdots + \mu_m$,
where each $\mu_i\in B_d \setminus \{0\}$ is a monomial of the form
$ \mu_i = b_i x_1^{e_{i1}} \dots x_n^{e_{in}} $
with $b_i \in B_0$ and $e_{ij} \in \Nat$.
We have $\deg( \mu_1 + \cdots + \mu_m ) = \max_i \deg( \mu_i )$,
so part~(2) of \ref{dklfja;skldjf} implies that 
$\delta(x) = \delta( \mu_1 + \cdots + \mu_m ) \le \max_i \delta( \mu_i )$.
So it's enough to show that $\delta( \mu_i ) \le M$ for all $i$.
Part~(1) of \ref{dklfja;skldjf} gives
$\delta(\mu_i) \le
\max\big( \delta( b_i), \delta( x_1),  \dots, \delta(x_n) \big)$,
so $\delta(\mu_i) \le M$ and we are done.
\end{proof}

\begin{corollary}  \label{dp9r8123849q23kjd}
Let $G$ be a totally ordered abelian group,
$B = \bigoplus_{i\in G} B_i$ a $G$-graded integral domain
of characteristic zero and
$\deg : B \to G \cup \{ -\infty \}$ the degree function determined
by the grading.
Assume:
\begin{enumerate}

\item $B$ has finite transcendence degree over a field $\bk$

\item $B$ is finitely generated as a $B_0$-algebra.

\end{enumerate}
Then $\deg$ is tame over $\bk$.

More precisely,
given any choice of $z_1,\dots,z_m\in B_0$ and homogeneous
$x_1,\dots,x_n\in B$
such that $B_0$ is algebraic over $\bk[ z_1, \dots, z_m ]$ and 
$ B = B_0 [ x_1, \dots, x_n ] $,
$$
\deg(D)= \max \{
\delta_D(z_1), \dots, \delta_D(z_m),  \delta_D(x_1), \dots, \delta_D(x_n) \}
\quad \text{for all $D \in \Der_\bk(B)$.}
$$
\end{corollary}

\begin{proof}
As $\bk$ is necessarily included in $B_0$, this is \ref{dkfjaiosdfupaio}
with $A=\bk$.
\end{proof}

The next two results are consequences of \ref{dp9r8123849q23kjd}.

\begin{corollary}  \label{dif9w8er093409u}
Let $\bk$ be a field of characteristic zero,
$B$ a $\bk$-affine integral domain and 
$G$ a totally ordered abelian group.
If $\deg : B \to G \cup \{ -\infty \}$ is the degree function
determined by some $G$-grading of $B$, then $\deg$ is tame over $\bk$.

More precisely,
given any choice of homogeneous elements $x_1,\dots,x_n\in B$
satisfying $ B = \bk [ x_1, \dots, x_n ] $,
we have $\deg(D) = \max_{1 \le i \le n} \delta_D( x_i )$
for all $D \in \Der_\bk(B)$.
\end{corollary}

\begin{proof}
Fix a grading $B = \bigoplus_{i \in G} B_i$ which determines $\deg$
and note that $\bk \subseteq B_0$.
Given homogeneous elements $x_1,\dots,x_n\in B$
satisfying $ B = \bk [ x_1, \dots, x_n ] $, it is certainly the case
that $ B = B_0 [ x_1, \dots, x_n ] $.
We may also choose $z_1, \dots, z_m \in B_0$ such that each $z_i$
is a monomial of the form $x_1^{e_{i1}} \cdots x_n^{e_{in}}$
($e_{ij} \in \Nat$) and $B_0$ is algebraic over $\bk[ z_1, \dots, z_m ]$.
Let $D \in \Der_\bk(B)$, then
$\deg(D)= \max \{
\delta_D(z_1), \dots, \delta_D(z_m),  \delta_D(x_1), \dots, \delta_D(x_n) \}$
by \ref{dp9r8123849q23kjd}.
Part~(1) of \ref{dklfja;skldjf} gives
$\delta_D(z_i) \le \max_{ 1 \le j \le n } \delta_D(x_j)$,
so $\deg(D) = \max_{1 \le i \le n} \delta_D( x_i )$.
\end{proof}

\begin{corollary}  \label{dddddidufqwiejll}
Let $R$ be a domain of finite transcendence degree over a field $\bk$
of characteristic zero and let $B = R[X_1, \dots, X_n] = R^{[n]}$.
Let $G$ be a totally ordered abelian group and define a $G$-grading
on $B$ by choosing $(d_1, \dots, d_n) \in G^n$ and  declaring that
the elements of $R \setminus \{0\}$ are homogeneous of degree $0$ and that
(for each $i$) $X_i$ is homogeneous of degree $d_i$.
Let $\deg : B \to G \cup \{ -\infty \}$ be the degree function determined by
this grading.
Then $\deg$ is tame over $\bk$.

More precisely, if $z_1, \dots, z_m \in R$ are such that $R$ is algebraic
over $\bk[ z_1, \dots, z_m ]$, then
$ \deg(D)= \max \{
\delta_D(z_1), \dots, \delta_D(z_m),  \delta_D(X_1), \dots, \delta_D(X_n) \}$
for every $D \in \Der_\bk(B)$.
\end{corollary}

\begin{proof}
Let $B = \bigoplus_{ i \in G } B_i$ be the grading and choose
$\xi_1, \dots, \xi_N \in B_0$ such that each $\xi_i$
is a monomial of the form $X_1^{e_{i1}} \cdots X_n^{e_{in}}$
($e_{ij} \in \Nat$) and $B_0$ is algebraic over $R[ \xi_1, \dots, \xi_N ]$;
then $B_0$ is algebraic over $\bk[ z_1, \dots, z_m, \xi_1, \dots, \xi_N ]$
and $B = B_0[ X_1, \dots, X_n ]$.
If $D \in \Der_\bk(B)$ then, by \ref{dp9r8123849q23kjd},
$$
\deg(D) = \max\{ 
\delta_D( z_1 ), \dots, \delta_D( z_m ), 
\delta_D( \xi_1 ), \dots, \delta_D( \xi_N ), 
\delta_D( X_1 ), \dots, \delta_D( X_n ) \}.
$$
Part~(1) of \ref{dklfja;skldjf} gives
$\delta_D(\xi_i) \le \max_{ 1 \le j \le n } \delta_D(X_j)$,
so we are done.
\end{proof}

Paragraphs \ref{dfqp29w3eoawjkk} and \ref{difuq90weukj}
are simple observations about localization of degree functions.
These facts are used in the proofs of\footnote{Lemma \ref{difuq90weukj}
would also be used for proving \ref{dfjqio23ejqlkmskkkd},
but this proof is omitted.}
\ref{doifq90348rqiojkdf},
\ref{difu;awesja;lk;lj} and \ref{d9320495823erfrj}.
Note that \ref{difuq90weukj} appeared in \cite{NurThesis}.

\begin{parag}  \label{dfqp29w3eoawjkk}
Let $B$ be a domain, $G$ a totally ordered abelian group and
$\deg : B \to G \cup \{ -\infty \}$ a degree function.
If $S \subseteq B \setminus \{0\}$ is a multiplicative set,
then $\deg$ has a unique extension to a degree function
$\DEG : S^{-1}B \to G \cup \{ -\infty \}$.
% $$
% \DEG : S^{-1}B \to G \cup \{ -\infty \}.
% $$
Indeed, it is easily verified that the map $\DEG$ defined
by $\DEG(0) = -\infty$ and
$\DEG( x/s ) = \deg(x) - \deg(s)$
(for $x \in B \setminus \{0\}$ and $s \in S$)
is a well-defined degree function and is the unique extension of $\deg$.
\end{parag}

\begin{lemma}  \label{difuq90weukj}
Let $B$ be a domain of characteristic zero,
$S \subseteq B \setminus \{0\}$ a multiplicative set,
$G$ a totally ordered abelian group, and
$\deg : B \to G \cup \{ -\infty \}$ and 
$\DEG : S^{-1}B \to G \cup \{ -\infty \}$
degree functions such that $\deg$ is the restriction of $\DEG$.
Consider $D \in \Der(B)$ and its extension $S^{-1}D \in \Der( S^{-1}B )$.
Then $\deg(D)$ is defined if and only if $\DEG(S^{-1}D)$ is defined,
and if both degrees are defined then they are equal.
\end{lemma}

\begin{proof}
As $\delta_D : B \to G \cup \{ -\infty \}$ 
is the restriction of  $\delta_{S^{-1}D} : S^{-1}B \to G \cup \{ -\infty \}$,
we have $U \subseteq U'$, where we define
$U = \setspec{ \delta_D(x) }{ x \in B }$ and
$U' = \setspec{ \delta_{S^{-1}D}(x) }{ x \in S^{-1}B }$.
We first observe that if $s \in S$ then 
\begin{multline}
\label{difuqealkjlaks}
\delta_{S^{-1}D}( 1/s )
= \DEG \big( (S^{-1}D) (1/s) \big) - \DEG( 1/s )  \\
= \DEG( -D(s)/s^2 ) - \DEG( 1/s )
= \deg( D s ) - 2 \deg(s) + \deg( s ) 
= \delta_D(s) .
\end{multline}
Applying part (1) of \ref{dklfja;skldjf} to $\delta_{S^{-1}D}$
gives, for any $x \in B$ and $s \in S$,
$$
\delta_{S^{-1}D}( x/s )
= \delta_{S^{-1}D}( x (1/s) )
\le \max( \delta_{S^{-1}D}(x), \delta_{S^{-1}D}(1/s) )
= \max( \delta_D(x), \delta_D(s) ) \in U.
$$
This shows that $\forall_{u' \in U'}\,\exists_{u \in U}\, u' \le u$.
This, together with $U \subseteq U'$, proves the Lemma.
\end{proof}

%  \begin{parag}  \label{aeuijkdjui37487i2009029012}
%  {\bf Proof of \ref{dfjqio23ejqlkmskkkd}.}
%  Let $D \in \Der_\bk(B)$.
%  Then $S^{-1}D \in \Der_\bk(S^{-1}B)$, so $\DEG( S^{-1}D )$ is defined
%  (because $\DEG$ is $\bk$-tame by assumption),
%  so $\deg(D)$ is defined, by \ref{difuq90weukj}.
%  \hfill \qedsymbol
%  \end{parag}

Recall that if $B$ is a domain of characteristic zero then
each  locally nilpotent derivation $\Delta : B \to B$ determines
a degree function $\deg_\Delta : B \to \Nat \cup \{ -\infty \}$
(cf.\ for instance \cite[1.1.7]{Freud:Book}).

\begin{corollary}  \label{doifq90348rqiojkdf}
Let $B$ be a domain of finite transcendence degree over a
field $\bk$ of characteristic zero.
Let $\deg_\Delta : B \to \Nat \cup \{ -\infty \}$
be the degree function determined by a locally nilpotent derivation
$\Delta: B \to B$.
Then $\deg_\Delta$ is tame over $\bk$.

Moreover, if $t \in B$ is such that $\Delta(t) \neq 0$ and 
$\Delta^2(t) = 0$, and
$z_1, \dots, z_m \in \ker\Delta$ are such that $\ker\Delta$ is algebraic
over $\bk[ z_1, \dots, z_m ]$, then for each $D \in \Der_\bk(B)$
\begin{equation}  \label{dfql23748132o192kj}
\deg_\Delta(D)
= \max\{ \delta_D( z_1 ), \dots, \delta_D( z_m ), \delta_D( t ) \}.
\end{equation}
% Moreover, if $t \in B$ is a preslice of $\Delta$ and
% $z_1, \dots, z_m \in \ker\Delta$ are such that $\ker\Delta$ is algebraic
% over $\bk[ z_1, \dots, z_m ]$, then 
% $\deg_\Delta(D)
% = \max\{ \delta_D( z_1 ), \dots, \delta_D( z_m ), \delta_D( t ) \}$.
\end{corollary}

\begin{proof}
Let $A = \ker\Delta$, $\alpha = \Delta(t) \in A \setminus \{0\}$
and $S = \setspec{ \alpha^n }{ n \in \Nat }$.
Then $S^{-1}B = (S^{-1}A)[t] = (S^{-1}A)^{[1]}$
(cf.\ for instance \cite[p.\ 27]{Freud:Book}).
Moreover, if $\DEG : S^{-1}B \to \Nat \cup \{ -\infty \}$ denotes the 
$t$-degree then $\deg_\Delta$ % : B \to \Nat \cup \{ -\infty \}$
is the restriction of $\DEG$, so the hypothesis of \ref{difuq90weukj}
is satisfied.
Apply either \ref{dp9r8123849q23kjd} or \ref{dddddidufqwiejll} to 
$\DEG$ and $S^{-1}D \in \Der_\bk( S^{-1}B )$:
as $S^{-1}A$ is algebraic over $\bk[ z_1, \dots, z_m ]$,
% $\DEG( S^{-1}D )$ is defined and 
$
\DEG( S^{-1}D )
= \max\{ \delta_{S^{-1}D}( z_1 ), \dots, \delta_{S^{-1}D}( z_m ),
\delta_{S^{-1}D}( t ) \}
$.
We have $\deg_\Delta(D) = \DEG( S^{-1}D )$ by \ref{difuq90weukj},
so we are done.
\end{proof}

\begin{Remark}
Let the notations and assumptions be as in \ref{doifq90348rqiojkdf}.
Then \eqref{dfql23748132o192kj} can be rewritten
(thanks to \ref{difjasdkjfa;klsd}) as 
$$
\deg_\Delta(D) = \max \big\{ \delta_D ( \ker\Delta ),\,  \delta_D(t) \big\}.
$$
However,
if we suppose that $\ker D \neq \ker\Delta$ then
Cor.~2.16 on p.~42 of \cite{Freud:Book} asserts that
$\deg_\Delta(D) =  \delta_D \big( \ker\Delta \big)$;
this last claim is not correct, as shown by the following example.
Let $B = \bk[z,t] = \kk2$,
$\Delta = \frac{\partial}{\partial t}$
and $D = z \frac{\partial}{\partial z} + t^2 \frac{\partial}{\partial t}$.
As $\Delta(t) \neq0$, $\Delta^2(t) =0$ and $\ker\Delta=\bk[z]$,
\eqref{dfql23748132o192kj} gives 
$\deg_\Delta(D) = \max\{ \delta_D(z), \delta_D(t) \} = \max\{ 0, 1 \} = 1$,
while $\delta_D \big( \ker\Delta \big)= \delta_D \big( \bk[z] \big)= 0$.
\end{Remark}

Here is another common situation where \ref{difuq90weukj} is useful
(compare with \ref{dfjqio23ejqlkmskkkd}):

\begin{corollary} \label{difu;awesja;lk;lj}
Let $L = \bk[ X_1^{\pm1}, \dots,  X_n^{\pm1} ]$ be the ring of
Laurent polynomials in $n$ variables over a field $\bk$
of characteristic zero,
let $\ggoth$ be a $G$-grading of $L$ where $G$ is some totally
ordered abelian group, and let 
$\deg_\ggoth : L \to G \cup \{ -\infty \}$ be the degree function
determined by $\ggoth$.
Let $B$ be a ring such that 
$\bk[ X_1, \dots,  X_n ] \subseteq B \subseteq L$
and let $\deg : B \to G \cup \{ -\infty \}$ be the restriction
of $\deg_\ggoth$.  Then $\deg$ is tame over $\bk$.
Moreover, if we also assume that each $X_i$ is a $\ggoth$-homogeneous
element of $L$ then
$\deg(D) = \max_{1 \le i \le n} \delta_D( X_i )$ for all $D \in \Der_\bk(B)$.
\end{corollary}

\begin{proof}
Let $S =
\setspec{ X_1^{e_1} \cdots X_n^{e_n} }{ (e_1, \dots, e_n) \in \Nat^n }$
and note that $S^{-1}B=L$.
Let $D \in \Der_\bk(B)$.
By \ref{dif9w8er093409u}, $\deg_\ggoth( S^{-1}D )$ is defined;
by \ref{difuq90weukj}, it follows that $\deg(D)$ is defined and
$\deg(D) = \deg_\ggoth( S^{-1}D )$; in particular, $\deg$ is tame over $\bk$.
Under the additional assumption that each $X_i$ is homogeneous,
\ref{dif9w8er093409u} gives
\begin{multline*}
\deg_\ggoth( S^{-1}D ) =
\max\{ \delta_{S^{-1}D}( X_1 ),  \delta_{S^{-1}D}( 1/X_1 ), \dots, 
\delta_{S^{-1}D}( X_n ),  \delta_{S^{-1}D}( 1/X_n ) \} \\
= \max\{ \delta_{D}( X_1 ), \dots, \delta_{D}( X_n ) \} ,
\end{multline*}
where for the last equality we used that 
$\delta_{S^{-1}D}( 1/X_i ) = \delta_D( X_i )$
for each $i$ (see \eqref{difuqealkjlaks} in the proof of \ref{difuq90weukj}).
We already noted that $\deg(D) = \deg_\ggoth( S^{-1}D )$, so we are done.
\end{proof}

\section*{Finite generation of the associated graded ring}

We shall now study triples $(B,G,\deg)$ as in \ref{setup} which satisfy
the additional condition that $\Gr B$ is a finitely generated algebra
over a zero-subring (as explained in \ref{dfiqpw9e3r90ad}, below).
For this type of consideration, the following device is useful.

\begin{definition}   \label{sdp091231u23ikjd}
\newcommand{\bA}{{\bar A}}
Let $(B,G,\deg)$ be as in \ref{setup}.
\begin{enumerate}
\setlength{\itemsep}{1mm}

\item By a \textit{subpair\/} of $(B,\Gr B)$,
we mean a pair $(A,\bA)$ where $A$ is a subset of $B$, $1 \in A$,
$\bA$ is a homogeneous subring of $\Gr B$ and:
\begin{equation}  \label{pair}
\tag{$\dagger$}
\textrm{Each homogeneous element of $\bA$ is of the form $\gr(a)$
for some $a\in A$.}
\end{equation}

\item Let $D \in \Der(B)$.
By a \textit{$D$-subpair} of $(B,\Gr B)$, we mean 
a subpair $(A,\bA)$ of $(B,\Gr B)$ such that $\delta_D(A)$ is defined.

\item If $(A,\bA)$ is a subpair of $(B,\Gr B)$ and $x\in B$, we define
$$
(A,\bA)_x = (A_x, \bA[\gr(x)]),
$$
where $A_x$ is the collection of all elements $b \in B$ which can be written
in the form  $b = \sum_{i=0}^m a_ix^i$ for some $m\in\Nat$ and
$a_0, \dots, a_m \in A$ satisfying:
\begin{equation}\label{Ax}
\tag{$\ddagger$}
\deg(a_jx^j) = \deg(b) \quad\text{for all $j$ such that $a_j\neq0$.}
\end{equation}

% \item If $(A,\bA)$ is a subpair of $(B,\Gr B)$ and $x\in B$, we define
% $$
% (A,\bA)_x = (A_x, \bA[\gr(x)]),
% $$
% where $A_x$ is the collection of all quantities $\sum_{i=0}^m a_ix^i$
% (with $m\in\Nat$, $a_i\in A$) satisfying
% %
% \begin{equation}\label{Ax}
% \tag{$\ddagger$}
% \deg(a_jx^j) = \deg({\textstyle\sum_{i=0}^m a_ix^i})
% \quad\text{for all $j$ such that $a_j\neq0$.}
% \end{equation}

\end{enumerate}
\end{definition}

\begin{Remark} \label{dkjfoq8w3283i8fk}
\newcommand{\bA}{{\bar A}}
As in part~(3) of \ref{sdp091231u23ikjd}, consider a subpair
$(A,\bA)$ of $(B,\Gr B)$ and $x\in B$.
Then $A \cup \{x\} \subseteq A_x \subseteq R[x]$,
where $R$ is the subring of $B$ generated by $A$.
\end{Remark}

\begin{lemma}   \label{sdkjfpqiowepakj999}
\newcommand{\bA}{{\bar A}}
Let $(B,G,\deg)$ be as in \ref{setup}.
\begin{enumerate}
\setlength{\itemsep}{1mm}

\item  If $(A,\bA)$ is a subpair of $(B,\Gr B)$,
then so is $(A,\bA)_x$ for each $x \in B$.

\item  Let $D \in \Der(B)$.  If $(A,\bA)$ is a $D$-subpair of $(B,\Gr B)$,
then so is $(A,\bA)_x$ for each $x \in B$.
Moreover, $\delta_D( A_x ) = \delta_D( A \cup \{x\} )$.

\end{enumerate}
\end{lemma}

\begin{proof}
\newcommand{\bx}{{\gr(x)}}
\newcommand{\by}{{\bar y}}
\newcommand{\bA}{{\bar A}}
\newcommand{\ba}[1]{{\bar a}_{#1}}
Let $(A,\bA)$ be a subpair of $(B,\Gr B)$, let $x \in B$,
and consider $(A,\bA)_x = (A_x, \bA[\gr(x)])$;
we show that $(A,\bA)_x$ is a subpair of $(B,\Gr B)$.
We may assume that $x\neq0$, because $(A,\bA)_0 = (A, \bA)$.
As $1 \in A$ and $A \subseteq A_x$, we have $1 \in A_x$.
% Write $\bx = \gr(x)$.
We have to show that if $\by$ is a homogeneous element of $\bA[\bx]$ then
$\by = \gr(y)$ for some $y\in A_x$.  Note that this is clear if
$\by=0$ (because \eqref{pair} implies $0\in A$, hence $0\in A_x$),
so assume $\by\neq0$. We have
$$
\by = \sum_{i=0}^m \ba i\bx^i
$$
for some $m \in \Nat$
and some homogeneous elements $\ba 0,\dots,\ba m \in \bA$ satisfying
\begin{equation}\label{bAx}
\deg(\ba j\bx^j) = \deg(\by) \qquad
\text{for all $j$ such that $\ba j\neq0$}.
\end{equation}
By \eqref{pair},
there exist $a_0,\dots,a_m\in A$ such that $\gr(a_i) = \ba i$
for all $i$.
Since $\deg(a_j x^j) = \deg(\gr( a_j x^j)) = \deg(\ba j \bx^j)$,
\eqref{bAx} implies that
$\deg(a_j x^j) = \deg(\by)$ whenever $a_j\neq0$.
Consequently,
$$
\by = \sum_{i=0}^m \gr (a_i x^i)
= \begin{cases}
\gr( \textstyle \sum_{i=0}^m a_ix^i ), & \text{if }
			\deg(  \textstyle \sum_{i=0}^m a_ix^i ) =  \deg(\by), \\
0, & \text{if }
			\deg(  \textstyle \sum_{i=0}^m a_ix^i ) <  \deg(\by).
\end{cases}
$$
Since $\by\neq0$, it follows that 
$\deg(  \textstyle \sum_{i=0}^m a_ix^i ) =  \deg(\by)$
(so $\sum_{i=0}^m a_ix^i \in A_x$) and that 
$\by= \gr\big( \sum_{i=0}^m a_ix^i \big)$,
so $\by = \gr(y)$ for some $y\in A_x$.
So $(A,\bA)_x$ is indeed a subpair of $(B,\Gr B)$,
and assertion~(1) is proved.

Let $D \in \Der(B)$, assume that $(A,\bA)$ is a $D$-subpair of $(B,\Gr B)$
and let $x \in B$.
To show that $(A,\bA)_x$ is a $D$-subpair of $(B,\Gr B)$, 
we have to show that $\delta_D( A_x )$ is defined.
We may assume that $x \neq 0$.
Let $y\in A_x$; then we may write 
$y = {\textstyle\sum_{i=0}^m a_i x^i}$ for some $m \in \Nat$ and
$a_0,\dots,a_m\in A$ such that \eqref{Ax} holds,
i.e., 
$$
\deg(a_jx^j) = \deg y \textrm{\ \ whenever\ } a_j\neq0.
$$
Write $f(X) = {\textstyle\sum_{i=0}^m a_i X^i}$; then
$y=f(x)$ and $Dy = f^{(D)}(x) + f'(x)Dx$.

If $j$ is such that $a_j\neq0$ then
$$
\deg( D(a_j)x^j ) = \delta_D(a_j)+\deg(a_j)+\deg(x^j)
= \delta_D(a_j)+\deg(a_jx^j)
= \delta_D(a_j)+\deg(y),
$$
so $\deg(f^{(D)}(x)) \le \delta_D(\alpha)+\deg(y)$ for some $\alpha \in A$.
Also, if $j>0$ is such that $a_j\neq0$ then
$$
\deg(j a_j x^{j-1} D(x))
= \deg(a_j x^j) -\deg(x)+\deg D(x)
= \deg(y) + \delta_D(x),
$$
so $\deg( f'(x)D(x) ) \le \deg(y) + \delta_D(x)$. Thus,
\begin{multline*}
\delta_D(y)+\deg(y) = \deg( D(y) )
\le \max( \deg(f^{(D)}(x)), \deg( f'(x)D(x) ) ) \\
\le \max( \delta_D(\alpha)+\deg(y),  \deg(y) + \delta_D(x))
\end{multline*}
and it follows that $\delta_D(y) \le \max(\delta_D(\alpha),\delta_D(x))$.
We have shown that $A_x \preceq_D A \cup \{x\}$.
As $1 \in A$, we have $x \in A_x$ and hence $A \cup \{x\} \subseteq A_x$.
Thus
\begin{equation}  \label{dfuq23984-13io}
\delta_D(A_x) = \delta_D( A \cup \{x\} ),
\end{equation}
by \ref{dkfjasiodufpao}.
As $\delta_D(A)$ is defined, so is $\delta_D(A \cup \{x\})$;
so, by \eqref{dfuq23984-13io}, $\delta_D(A_x)$ is defined.
This proves assertion~(2).
\end{proof}

\begin{lemma}    \label{dfiqpw9e3r90ad}
Given $(B,G,\deg)$ as in \ref{setup}, the following hold.
\begin{enumerate}

\item \it For any $0$-subring $Z$ of $B$ (cf.\ \ref{odfuipq90w389j}),
$\Gr B$ is a $Z$-algebra.

\item \label{dkfjq234729qijk}
\it If $\deg(x) \ge 0$ for all $x \in B \setminus \{0\}$ then
the subring $Z = \setspec{ x \in B }{ \deg(x) \le 0 }$ of $B$ is in fact
a $0$-subring of $B$, and is factorially closed in $B$.
By {\rm (1)}, it follows that $\Gr B$ is a $Z$-algebra.

\item  \it If $\setspec{ \deg(x) }{ x \in B \setminus \{0\} }$
is a well-ordered subset of $G$
then $\deg(x) \ge 0$ for all $x \in B \setminus \{0\}$, i.e., 
the hypothesis of \eqref{dkfjq234729qijk} is satisfied.

\end{enumerate}
\end{lemma}

\begin{proof}
Let $Z$ be a $0$-subring of $B$.
If $B_i$, $B_{i^-}$ and $B_{[i]}$ are defined as in
\ref{dfjapisdjfa;kj} then the composite
$
Z \hookrightarrow B_0 \to B_{[0]} \hookrightarrow \Gr B
$
is an injective homomorphism of rings $Z \to \Gr B$, $z \mapsto \gr(z)$.
This defines a structure of $Z$-algebra on $\Gr B$:
if $z \in Z$ and $\xi \in \Gr B$, then $z\xi = \gr(z)\xi$.
Assertions (2) and (3) are trivial.
\end{proof}

\begin{lemma}   \label{d9f8q9weij8932}
Let $(B,G,\deg)$ be as in \ref{setup} and let $Z$ be a $0$-subring of $B$.
Assume that $\Gr(B)$ is finitely generated as a $Z$-algebra
(cf.\ \ref{dfiqpw9e3r90ad}) and consider elements $x_1,\dots,x_n\in B$
satisfying $\Gr B = Z [ \gr(x_1), \dots, \gr(x_n) ]$.
\begin{enumerate}
\setlength{\itemsep}{1mm}

\item There exists a set $E$ satisfying
$Z \cup \{ x_1, \dots, x_n \} \subseteq E \subseteq Z[x_1, \dots, x_n]$ and:
\begin{enumerate}
\setlength{\itemsep}{1mm}

\item $\forall_{x\in B\setminus\{0\}}\,\exists_{e\in E}\, \deg(x-e)<\deg x$

\item $\displaystyle \delta_D(E) = \delta_D(Z \cup \{ x_1, \dots, x_n \} )$
for every $D \in \Der(B)$ such that $\delta_D( Z )$ is defined.
% $$
% \displaystyle \delta_D(E) = \delta_D(Z \cup \{ x_1, \dots, x_n \} ).
% $$

\end{enumerate}

\item If $\setspec{ \deg(x) }{ x \in B \setminus \{0\} }$ is
a well-ordered subset of $G$ then $B = Z[ x_1, \dots, x_n ]$.  

\end{enumerate}
\end{lemma}

\begin{proof}
\newcommand{\bx}{{\bar x}}
\newcommand{\by}{{\bar y}}
\newcommand{\bA}{{\bar A}}
\newcommand{\ba}[1]{{\bar a}_{#1}}
%%%
We have $\Gr B = Z[\bx_1,\dots,\bx_n ]$, where we define
$\bx_i = \gr(x_i)$ for all $i$. 
Define $A_0 = Z \subseteq B$ and $\bA_0 =  %  \bar Z \subseteq \Gr B$,
\setspec{ \gr(z) }{ z \in Z } \subseteq \Gr B$,
and note that $(A_0,\bA_0)$ is
a subpair of $(B,\Gr B)$.
For $1\le i\le n$, define
$(A_{i},\bA_{i}) = (A_{i-1},\bA_{i-1})_{x_i}$; then set $E=A_n$ and
note that
$Z \cup \{ x_1, \dots, x_n \} \subseteq E \subseteq Z[x_1, \dots, x_n]$,
by \ref{dkjfoq8w3283i8fk}.
Also, $(A_n,\bA_n)$ is (by \ref{sdkjfpqiowepakj999})
a subpair of $(B,\Gr B)$ and
$\bA_n  = \bA_0[\bx_1,\dots,\bx_n ] = Z[\bx_1,\dots,\bx_n ] = \Gr B$;
it follows that each homogeneous element of $\Gr(B)=\bA_n$ is of the form
$\gr(e)$ for some $e \in E=A_n$; so $E$ satisfies condition~(a). 
Let  $D \in \Der(B)$ be such that $\delta_D( Z )$ is defined;
then $(A_0,\bA_0)$ is a $D$-subpair of $(B,\Gr B)$;
so, by repeated application of \ref{sdkjfpqiowepakj999},
$(A_n,\bA_n)$ is a $D$-subpair of $(B,\Gr B)$
and $\delta_D(E) = \delta_D(A_n) = \delta_D(Z \cup \{ x_1, \dots, x_n \})$.
So $E$ satisfies~(b).

We prove (2) by contradiction: assume that
$\setspec{ \deg(x) }{ x \in B \setminus \{0\} }$ is
well-ordered and $B \neq Z[x_1, \dots, x_n]$.
Pick $b_0 \in B \setminus Z[x_1, \dots, x_n]$ such that
$\deg( b_0 )$ is the least element of 
$\setspec{ \deg(x) }{ x \in B \setminus Z[x_1, \dots, x_n] }$.
Then there exists $e \in E \subseteq Z[x_1, \dots, x_n]$ such
that $\deg( b_0 - e ) < \deg( b_0 )$, and this leads to a contradiction.
So $B = Z[x_1, \dots, x_n]$.
\end{proof}

\begin{Remark}
The assumption that $\setspec{ \deg(x) }{ x \in B \setminus \{0\} }$ is
well-ordered, in \ref{d9f8q9weij8932}(2), is needed.
Indeed, consider $B = \bk[x,y]$ and $\deg : B \to \Integ \cup \{ -\infty \}$ 
as in the proof of \ref{dfkljasdklfja}.
Then  $\deg(x)=1$ and $\deg( y - a_0 ) = -k$ where $k\ge1$.
Define $x_1 = x$ and $x_2 = x^{2k-1}(y - a_0)^2$,
then $\deg(x_1)=1$ and $\deg(x_2)=-1$,
so $\Gr(B) = \bk[ \gr(x_1), \gr(x_2) ]$ (because $\Gr(B) \isom \bk[t,t^{-1}]$).
However, $B \neq \bk[x_1, x_2]$.
\end{Remark}

\begin{proposition}  \label{NEW_dofiq-290348rakj}
Let $(B,G,\deg)$ be as in \ref{setup} and suppose that
\begin{enumerate}

\item $\setspec{ \deg(x) }{ x \in B \setminus \{0\} }$ is a well-ordered
subset of $G$

\item $\Gr B$ is finitely generated as a $Z$-algebra,

\end{enumerate}
where $Z = \setspec{ x \in B }{ \deg(x) \le 0 }$ (cf.\ \ref{dfiqpw9e3r90ad}).
Then the following hold:
\begin{enumerate}
\addtocounter{enumi}{2}

\item $B$ is finitely generated as a $Z$-algebra;

\item  \label{dfkjq2i3uqkqwj9o}
if $A$ is a subring of $Z$ such that $\trdeg_A(Z) < \infty$,
then $\deg$ is tame over $A$.
% then $\deg(D)$ is defined for every $D \in \Der_A(B)$.

\end{enumerate}
More precisely, let $A$ be as in \eqref{dfkjq2i3uqkqwj9o} and let
$z_1,\dots,z_m\in Z$ and $x_1,\dots,x_n\in B$
be such that $Z$ is algebraic over $A[ z_1, \dots, z_m ]$ and 
$ \Gr B = Z [ \gr(x_1), \dots, \gr(x_n) ] $;
then $B = Z[x_1, \dots, x_n]$ and
$$
\deg(D)= \max \{
\delta_D(z_1), \dots, \delta_D(z_m),  \delta_D(x_1), \dots, \delta_D(x_n) \},
\quad \text{for all $D \in \Der_A(B)$.}
$$
\end{proposition}

\begin{proof}
In view of \ref{dfiqpw9e3r90ad},
assumption~(1) implies that $Z$ is a $0$-subring of $B$ and hence
that $\Gr B$ is a $Z$-algebra, so assumption~(2) makes sense.
Let $z_1,\dots,z_m\in Z$ and $x_1,\dots,x_n\in B$ be
such that $Z$ is algebraic over $A[ z_1, \dots, z_m ]$ and 
$ \Gr B = Z [ \gr(x_1), \dots, \gr(x_n) ] $;
then $B = Z[x_1, \dots, x_n]$ by \ref{d9f8q9weij8932}.
Let $D \in \Der_A(B)$.
To prove the Proposition, we have to show that 
$\delta_D(x) \le M$ for all $x \in B$, where we define
$$
M = \max \{
\delta_D(z_1), \dots, \delta_D(z_m),  \delta_D(x_1), \dots, \delta_D(x_n) \} .
$$
Choose a subset $E \subseteq Z[ x_1,\dots,x_n ]$ satisfying the
requirements of \ref{d9f8q9weij8932}.
By \ref{difjasdkjfa;klsd}, $\delta_D(Z)$ is defined and is equal to
$\max_{1 \le i \le m} \delta_D( z_i )$;
so $E$ satisfies:
\begin{equation*}
\forall_{x\in B\setminus\{0\}}\,\exists_{e\in E}\, \deg(x-e)<\deg x 
\quad \text{and} \quad
\delta_D(E) = \delta_D( Z \cup \{ x_1, \dots, x_m \} ) = M .
\end{equation*}
% and $\delta_D(E) = \delta_D( Z \cup \{ x_1, \dots, x_m \} ) = M$.
By contradiction, assume that some $x \in B$ satisfies $\delta_D(x) > M$;
then the set
$
S_0 = \setspec{ i \in G }
{ \exists_{x \in B}\, \big( \deg(x)=i \text{\ and\ } \delta_D(x)>M \big) }
$
is not empty.
By assumption~(1), we may consider the least element $i_0$ of $S_0$.
Now pick $x \in B$ such that $\deg(x) = i_0$ and $\delta_D(x)>M$;
note in particular that  $\delta_D(x)>M$ and $\delta_D(E) = M$ imply
that $x \notin E$.
Choose $e \in E$ such that $\deg(x-e)<\deg(x)$
and note that $x-e \neq 0$; so $\deg(x-e)$ is an element of $G$ strictly
less than $i_0$.
By minimality of $i_0$, it follows that $\delta_D( x - e ) \le M$.

Note that $\deg(x) = \deg(e)$. 
If $\deg(Dx) = \deg (De)$, it follows immediately that 
$\delta_D(x) = \delta_D(e) \le M$, a contradiction;
so $\deg(Dx) \neq \deg (De)$ and consequently
$\deg( Dx - De ) = \max( \deg(Dx), \deg(De) )$. Then
\begin{multline*}
\delta_D(x) + \deg x = \deg( Dx ) \le\max( \deg(Dx), \deg(De) ) \\
= \deg D(x-e) = \delta_D(x-e) + \deg(x-e) \le M + \deg(x-e),
\end{multline*}
so $\delta_D(x) \le  M  +  \deg(x-e) - \deg x  <  M$, a contradiction.
\end{proof}

\begin{corollary}  \label{dofiq-290348rakj}
Let $B$ be an integral domain of finite transcendence degree over a field
$\bk$ of characteristic zero.
Suppose that $\deg: B \to G \cup \{ -\infty \}$
(where $G$ is a totally ordered abelian group)
is a degree function satisfying the conditions
\begin{enumerate}

\item $\setspec{ \deg(x) }{ x \in B \setminus \{0\} }$ is a well-ordered
subset of $G$,

\item $\Gr B$ is a finitely generated algebra over
the ring $Z = \setspec{ x \in B }{ \deg(x) \le 0 }$.

\end{enumerate}
Then $\deg$ is tame over $\bk$ and $B$ is finitely generated as a $Z$-algebra.
More precisely,
if $z_1,\dots,z_m\in Z$ and $x_1,\dots,x_n\in B$ are
such that $Z$ is algebraic over $\bk[ z_1, \dots, z_m ]$ and 
$ \Gr B = Z [ \gr(x_1), \dots, \gr(x_n) ] $,
then $B = Z[x_1, \dots, x_n]$ and
$$
\deg(D)= \max \{
\delta_D(z_1), \dots, \delta_D(z_m),  \delta_D(x_1), \dots, \delta_D(x_n) \},
\quad \text{for all $D \in \Der_\bk(B)$.}
$$
\end{corollary}

\begin{proof}
In view of \ref{dfiqpw9e3r90ad},
assumption~(1) implies that $Z$ is a $0$-subring of $B$ and hence
that $\Gr B$ is a $Z$-algebra, so assumption~(2) makes sense.
It is also noted in \ref{dfiqpw9e3r90ad} that $Z$ is factorially closed
in $B$; this implies that $\bk \subseteq Z$, so all hypotheses of
\ref{NEW_dofiq-290348rakj} are satisfied with $A=\bk$.
The result follows from \ref{NEW_dofiq-290348rakj}.
\end{proof}

\begin{corollary}  \label{;sdlkuf1238whd8787sd676}
Let $B$ be an integral domain containing a field $\bk$ of characteristic zero.
Suppose that $\deg: B \to G \cup \{ -\infty \}$
(where $G$ is a totally ordered abelian group)
is a degree function satisfying the conditions
\begin{enumerate}

\item $\setspec{ \deg(x) }{ x \in B \setminus \{0\} }$ is a well-ordered
subset of $G$,

\item $\Gr B$ is a finitely generated $\bk$-algebra.

\end{enumerate}
Then $\deg$ is tame over $\bk$ and $B$ is a finitely generated $\bk$-algebra.
More precisely, if $x_1,\dots,x_n\in B$ are such that 
$ \Gr B = \bk[ \gr(x_1), \dots, \gr(x_n) ] $,
then:
\begin{enumerate}
\addtocounter{enumi}{2}

\item $B = \bk[x_1, \dots, x_n]$

\item $\deg(D)= \max \{ \delta_D(x_1), \dots, \delta_D(x_n) \}$
for all $D \in \Der_\bk(B)$

\item $Z = \bk[z_1, \dots, z_m]$, where we define
$Z = \setspec{ x \in B }{ \deg(x) \le 0 }$
and where $z_1, \dots, z_m$ denote the elements of $\{ x_1, \dots, x_n \}$
of degree $0$.

\end{enumerate}
\end{corollary}

\begin{proof}
By \ref{dfiqpw9e3r90ad}, assumption~(1) implies that
$Z = \setspec{ x \in B }{ \deg(x) \le 0 }$ is a $0$-subring of $B$
(so $\Gr B$ is a $Z$-algebra) and is factorially closed in $B$.
The last condition implies that $\bk \subseteq Z$,
so $\Gr B$ is a $\bk$-algebra and assumption~(2) makes sense.
% It also follows that $\Gr B$ is finitely generated as a $Z$-algebra.

Let $x_1,\dots,x_n\in B \setminus \{0\}$ be such that
$\Gr B = \bk[ \gr(x_1), \dots, \gr(x_n) ]$.
As $\bk$ is a $0$-subring of $B$ (because $\bk \subseteq Z$),
\ref{d9f8q9weij8932} implies that $B = \bk[x_1, \dots, x_n]$.
% In particular $\trdeg_\bk(B)<\infty$, 
% so all hypotheses of \ref{dofiq-290348rakj} are satisfied.

Write $\Gr B = \bigoplus_{i \in G} B_{[i]}$
with notation as in \ref{dfjapisdjfa;kj}.
Let $\mu : Z \to B_{[0]}$ be the map
$Z = B_0 \to B_0/B_{0^-} = B_{[0]}$, and note
that $\mu$ is an isomorphism of $\bk$-algebras and 
$\mu(z) = \gr(z)$ for all $z \in Z$.
Let $z_1, \dots, z_m$ be the elements of $\{ x_1, \dots, x_n \}$
of degree $0$;
as $\Gr B = \bk[ \gr(x_1), \dots, \gr(x_n) ]$ where 
$\deg( \gr x_i ) = \deg(x_i) \ge 0$ for each $i$, it follows that 
$B_{[0]} = \bk[ \gr(z_1), \dots, \gr(z_m) ]$.
So the composite
$\bk[ z_1, \dots, z_m ] \hookrightarrow Z \xrightarrow{\mu} B_{[0]}$
is surjective, and consequently $Z = \bk[ z_1, \dots, z_m ]$.
All hypotheses of \ref{dofiq-290348rakj} are satisfied, so
$$
\deg(D) = \max
\{\delta_D(z_1), \dots, \delta_D(z_m),  \delta_D(x_1), \dots, \delta_D(x_n) \}
= \max \{ \delta_D(x_1), \dots, \delta_D(x_n) \} 
$$
for every $D \in \Der_\bk(B)$.
\end{proof}

%  \begin{subparag} \label{dfuqpwejdkl}
%  Recall that a \textit{valuation\/} of a ring $B$
%  is a set map $v : B \to G \cup \{ \infty \}$,
%  where $G$ is a totally ordered abelian group, satisfying for all $x,y \in B$:
%  (i)~$v(x)=\infty \Leftrightarrow x=0$, (ii)~$v(xy) = v(x)+v(y)$
%  and (iii)~$v(x+y) \ge \min( v(x), v(y) )$.
%  It is clear that if $v : B \to G \cup \{ \infty \}$ is a valuation then the
%  map $B \to G \cup \{ -\infty\}$, $x \mapsto -v(x)$, is a degree function;
%  and that if $\deg : B \to G \cup \{ -\infty\}$ is a degree function,
%  then $B \to G \cup \{ \infty \}$, $x \mapsto -\deg(x)$, is a valuation.
%  \end{subparag}

\begin{proposition}  \label{d9320495823erfrj}
Let $B$ an integral domain of finite transcendence degree over
a field $\bk$ of characteristic zero
and $\deg : B \to G \cup \{ -\infty\}$ a degree function,
where $G$ is a totally ordered abelian group.
Assume:
\begin{enumerata}

\item $\setspec{ \deg(x) }{ x \in B \setminus \{0\} }$ is a well-ordered
subset of $G$

\item  % $\trdeg_\bk (B)<\infty$ and
$\Frac(B)/\Frac(Z)$ is a one-dimensional function field,
where $Z$ denotes the subring $\setspec{ x \in B }{ \deg(x) \le 0 }$ of $B$.

\end{enumerata}
Then $\deg$ is tame over $\bk$.
Moreover, the ordered monoid
$\setspec{ \deg(x) }{ x \in B \setminus \{0\} }$
can be embedded in $(\Nat, +, \le)$.
\rien{ %%%%%%%%%%%%%%%%%%%%%%%%%%%%%%%%%%%%%%%%%%%%%%%%%%%%%%%
\begin{enumerata}
\addtocounter{enumi}{2}

\item The ordered monoid
$\setspec{ \deg(x) }{ x \in B \setminus \{0\} }$
can be embedded in $(\Nat, +, \le)$.

\item There exist finite subsets 
$\{ z_1, \dots, z_m \}$ of $Z$ and $\{ x_1, \dots, x_n \}$ of $B$
satisfying:
\begin{enumerate}

\item[(i)] $Z$ is algebraic over $\bk[ z_1, \dots, z_m ]$

\item[(ii)] for each $b \in B$,
there exists $z \in Z \setminus \{0\}$ such that $zb \in Z[ x_1, \dots, x_n ]$.
% and $\xi \in Z[ x_1, \dots, x_n ]$ such that $\deg(zb-\xi)< \deg(b)$.
\end{enumerate}

\item For any choice of 
$\{ z_1, \dots, z_m \}$ and $\{ x_1, \dots, x_n \}$ 
% satisfying {\rm (i)} and {\rm (ii)},
as in {\rm (d)},
$$
\deg(D)= \max \{
\delta_D(z_1), \dots, \delta_D(z_m),  \delta_D(x_1), \dots, \delta_D(x_n) \}
\text{ for all $D \in \Der_\bk(B)$.}
$$
\end{enumerata}
}  %%% end of rien %%%%%%%%%%%%%%%%%%%%%%%%%%%%%%%%%%%%%%%%%%%%%%%%
\end{proposition}

\begin{proof}
Let $S = Z \setminus \{0\}$, $B' = S^{-1}B$ and $Z' = S^{-1}Z = \Frac(Z)$.
By \ref{dfqp29w3eoawjkk}, $\deg$ extends to a 
degree function $\deg' : B' \to G \cup \{ -\infty\}$.
Note that
\begin{gather}
\label{di81923jlaksdnf}
\trdeg_\bk (B')<\infty, \\[1mm]
\label{aaapeiu;qakm}
\begin{minipage}[t]{.9\textwidth}
$\setspec{ \deg'(x) }{ x \in B' \setminus \{0\} }$
is equal to $\setspec{ \deg(x) }{ x \in B \setminus \{0\} }$
and hence is a well-ordered subset of $G$,
\end{minipage} \\[1mm]
\label{89374oijdf}
Z' = \setspec{ x \in B' }{ \deg'(x) \le 0 }.
\end{gather}

Let $L = \Frac(B)$ and, using \ref{dfqp29w3eoawjkk} again,
let $\DEG : L \to G \cup \{-\infty\}$
be the unique degree function which extends $\deg$ and $\deg'$.
Let $v : L \to  G \cup \{ \infty\}$ be the valuation of $L$
defined by $v(x) = -\DEG(x)$.
As $\deg(x)=0$ for all $x \in Z \setminus \{0\}$, we note that 
$v$ is a valuation over $Z'$; as $L/Z'$ is a one-dimensional function field,
it follows that $v$ is a rank $1$ discrete valuation; so the
residue field $\kappa$ of $v$ is a finite extension of $Z'$
and $\setspec{ v(x) }{ x \in L^* } \isom \Integ$.
It follows that
$\setspec{ \deg(x) }{ x \in B \setminus \{0\} }$
can be embedded in $(\Nat, +, \le)$.

\rien{ %%%%%%%%%%%%%%%%%%%%%%%%%%%%%%%%%%%%%%%%%%%%%
Write $\Gr(B) = \bigoplus_{i \in G} B_{[i]}$,
$\Gr(B') = \bigoplus_{i \in G} B'_{[i]}$
and $\Gr(L) = \bigoplus_{i \in G} L_{[i]}$ for the associated graded
rings determined by $(B,\deg)$, $(B',\deg')$ and $(L,\DEG)$ respectively;
then there are injective degree-preserving homomorphisms of graded rings,
$\Gr(B) \hookrightarrow \Gr(B') \hookrightarrow \Gr(L)$.
% and in particular $Z' = B'_{[0]} \subseteq L_{[0]} = \kappa$.
}  %%% end of \rien %%%%%%%%%%%%%%%%%%%%%%%%%%%%%%%%%%%%%%%%%%%%%
Consider the associated graded rings 
$\Gr(B)$, $\Gr(B') = \bigoplus_{i \in G} B'_{[i]}$ and $\Gr(L)$
determined by $(B,\deg)$,  $(B',\deg')$ and $(L,\DEG)$ respectively,
and note that $Z' = B'_{[0]} \subseteq \Gr(B')$.
As $\DEG$ extends $\deg'$ and $\deg'$ extends $\deg$,
there are injective degree-preserving homomorphisms of graded rings,
$\Gr(B) \hookrightarrow \Gr(B') \hookrightarrow \Gr(L)$.
Using that $v$ is a rank $1$ discrete valuation, 
we get $\Gr(L) \isom \kappa[ t, t^{-1} ]$ where $t$ is an indeterminate over
$\kappa$. Thus 
$$
Z' \subseteq \Gr(B') \subseteq \kappa[ t, t^{-1} ].
$$
Now $[\kappa : Z']<\infty$, so $\kappa[ t, t^{-1} ]$ is a finitely generated
$Z'$-algebra of transcendence degree $1$ over $Z'$;
it follows that
\begin{equation} \label{9r18239ria}
\text{$\Gr(B')$ is finitely generated as a $Z'$-algebra.}
\end{equation}
Let $D \in \Der_\bk(B)$, and consider $S^{-1}D \in \Der_\bk( B' )$.
By  \eqref{di81923jlaksdnf},  \eqref{aaapeiu;qakm}, \eqref{89374oijdf}
and \eqref{9r18239ria}, $(B', G, \deg')$ and $Z'$ satisfy the hypothesis of
\ref{dofiq-290348rakj} and consequently $\deg'(S^{-1}D)$ is defined;
by \ref{difuq90weukj}, $\deg(D)$ is defined.
So $\deg$ is tame over $\bk$.
\end{proof}

\begin{Remark}
Let us indicate how to compute the value of $\deg(D)$, in the above proof.
We have $\bk \subseteq Z$, because (\ref{dfiqpw9e3r90ad})
$Z$ is factorially closed in $B$; so we may choose $z_1, \dots, z_m \in Z$
such that $Z$ is algebraic over $\bk[z_1, \dots, z_m]$.
Also note that $\Gr(B) \hookrightarrow \Gr(B')$ is the localization:
$\Gr(B') = S^{-1}\Gr(B)$; this and \eqref{9r18239ria}
imply that we can choose $x_1, \dots, x_n  \in B$ satisfying
$\Gr(B') = Z' [ \gr(x_1), \dots, \gr(x_n) ]$.
As $Z'$ is algebraic over $\bk[ z_1, \dots, z_m ]$
and $\Gr(B') = Z' [ \gr(x_1), \dots, \gr(x_n) ]$, 
\ref{dofiq-290348rakj} gives
\begin{multline*}
\deg'( S^{-1}D )
= \max\{ \delta_{ S^{-1}D }( z_1 ), \dots, \delta_{ S^{-1}D }( z_m ), 
\delta_{ S^{-1}D }( x_1 ), \dots, \delta_{ S^{-1}D }( x_n ) \} \\
= \max\{ \delta_{ D }( z_1 ), \dots, \delta_{ D }( z_m ), 
\delta_{ D }( x_1 ), \dots, \delta_{ D }( x_n ) \} .
\end{multline*}
Now \ref{difuq90weukj} implies that $\deg(D) = \deg'( S^{-1}D )$,
so we conclude that
$$
\deg( D ) = \max\{ \delta_{ D }( z_1 ), \dots, \delta_{ D }( z_m ), 
\delta_{ D }( x_1 ), \dots, \delta_{ D }( x_n ) \} .
$$
\end{Remark}

\providecommand{\bysame}{\leavevmode\hbox to3em{\hrulefill}\thinspace}
\providecommand{\MR}{\relax\ifhmode\unskip\space\fi MR }
\providecommand{\MRhref}[2]{%
  \href{http://www.ams.org/mathscinet-getitem?mr=#1}{#2}
}
\providecommand{\href}[2]{#2}

% \bibliographystyle{amsplain}
% \bibliography{/home/ddaigle/articles/bib/dbase}
\end{document}